\documentclass[]{interact}

\usepackage[numbers,sort&compress]{natbib}

\theoremstyle{plain}
\newtheorem{theorem}{Theorem}[section]
\newtheorem{lemma}[theorem]{Lemma}

\newtheorem{proposition}[theorem]{Proposition}

\theoremstyle{definition}
\newtheorem{definition}[theorem]{Definition}

\newtheorem{assumption}[theorem]{Assumption}

\theoremstyle{remark}

\usepackage{amsmath}
\usepackage{amssymb}
\usepackage{amsthm}
\usepackage{algorithm}
\usepackage{algorithmic}
\usepackage{tikz,pgfplots}
\usetikzlibrary{matrix}
\usepackage{bbm}
\usepackage{bm}
\usepackage{xurl}
\usepackage{longtable}

\def\ZZ{\mathbb{Z}}
\def\RR{\mathbb{R}}

\def\NN{\mathbb{N}}
\def\EE{\mathbb{E}}
\def\PP{\mathbb{P}}
\def\FF{\mathbb{F}}
\def\GG{\mathbb{G}}

\def\F{{\cal F}}
\def\G{{\cal G}}
\def\R{{\cal R}}

\newcommand{\co}{\operatorname{co}}

\newcommand{\argmin}{\operatornamewithlimits{argmin}}
\DeclareMathOperator*{\var}{Var}
\DeclareMathOperator*{\dist}{dist}
\newcommand{\ind}{\mathbbm{1}}

\begin{document}
	
	\title{Perturbed Iterate SGD for Lipschitz Continuous Loss Functions with Numerical Error and Adaptive Step Sizes}

	\author{		
	\name{Michael R. Metel\thanks{michaelrmetel@gmail.com}}}	

	\maketitle
	
\begin{abstract}
Motivated by neural network training in finite-precision arithmetic environments, this work studies the convergence of perturbed iterate SGD using adaptive step sizes in an environment with numerical error. Considering a general stochastic Lipschitz continuous loss function, an asymptotic convergence result to a Clarke stationary point is proven as well as the non-asymptotic convergence to an approximate stationary point in expectation. It is assumed that only an approximation of the loss function's stochastic gradient can be computed, in addition to error in computing the SGD step itself. 
\end{abstract}
	


\section{Introduction}
\label{int}
This paper studies the convergence of perturbed iterate stochastic gradient descent (PISGD) using adaptive step sizes in an environment with numerical error. The assumptions are given in a general form, but are motivated by the error from using finite precision arithmetic while training neural networks. Given the continuously increasing size of deep learning models, there is a strong motivation to do training in lower-bit formats to enable more efficient training. The majority of research in this area is focused on hardware design, using a combination of number formats of different precision to accelerate training and reduce memory requirements, while aiming to incur minimal accuracy degradation, see \citep[Table 1]{wang2022}. Our work is complementary to this line of research, with a focus on modelling numerical error and attempting to adapt and extend the convergence analysis of PISGD using infinite precision, i.e., in $\RR^d$, to environments with numerical error. 

The convergence analyses, found in Sections \ref{perturbedSGD} and \ref{uncon_conv}, focus on finding an (approximate) stationary point of a function $f: \RR^d\rightarrow \RR\label{eq:0}$ which can be written as $f=\EE[F(\cdot,\bm{\xi})]\label{F}$ for a function $F:\RR^d\times\RR^n\rightarrow \RR$. The function $F(\cdot,\bm{\xi})$ is Lipschitz continuous, with the precise details given in Section \ref{LipFunc}, and $\bm{\xi}\in\RR^n\label{xi}$ is a random vector from a probability space $(\Omega,\F,P)$. Unlike assuming that $F(\cdot,\bm{\xi})$ is convex or that it has a Lipschitz continuous gradient, this assumption is much closer to reality as a wide range of neural network architectures are known to be at least locally Lipschitz continuous \citep{davis2020}. 

In a fixed finite-precision environment, it is generally not possible to prove convergence to a stationary point given that all such points may not even be representable, e.g., all stationary points could be irrational. The presented asymptotic convergence analysis, therefore, implicitly requires that the precision of representable numbers increases through time if it were to be ``implemented", such as by using a sequence of finite-precision environments of increasing precision over an infinite time horizon, with the rounding error decreasing to zero in the limit (see the paragraph before Theorem \ref{non-asy} for more details). However, this analysis, culminating in Theorem \ref{convergence}, still presents a novel convergence result to a Clarke stationary point with numerical error, while also serving as the foundation for the non-asymptotic convergence analyses, presented as Theorems \ref{non-asy} and \ref{non-asy-fixed}, where convergence is proven to an approximate stationary point in expectation after a predetermined number of iterations, which in principle can be implemented in a single fixed finite-precision environment. 

These convergence results are each proven for a subclass of a broad class of adaptive step sizes encompassing variants such as gradient normalization and gradient clipping. In Section \ref{empirical}, an example of our proposed class of adaptive step sizes is demonstrated on image recognition tasks in fixed-point arithmetic environments. Before these results, an overview of fixed-point arithmetic is given in Section \ref{finiteprec}, and past work studying optimization with numerical error is discussed in Section \ref{litreview}. The paper concludes in Section \ref{con}, with a table of notation given in Appendix \ref{notation}. 

\section{Lipschitz Continuous Loss Functions}
\label{LipFunc}
This section contains the required assumptions and resulting properties for $f$. It is assumed that $F(\cdot,\bm{\xi})$ is continuous for each $\bm{\xi}\in \RR^n$, and $F(\bm{w},\cdot)$ is Borel measurable for each $\bm{w}\in \RR^d\label{w}$. For almost all $\bm{\xi}\in\RR^n$, 
\begin{alignat}{6}
	|F(\bm{w},\bm{\xi})-F(\bm{w}',\bm{\xi})|\leq L_0(\bm{\xi})\|\bm{w}-\bm{w}'\|_2\nonumber
\end{alignat}
for all $\bm{w},\bm{w}'\in \RR^d$, where $L_0:\RR^n\rightarrow\RR\label{L0}$ is a measurable function which is square integrable, $Q:=\EE[L_0(\bm{\xi})^2]<\infty\label{Q}$. It follows that $f$ is $L_0:=\EE[L_0(\bm{\xi})]\label{EL0}$-Lipschitz continuous \citep[Proposition 2]{metel2022}. As is common for loss functions used in machine learning, we make the following assumption.
\begin{assumption}\label{pos_loss}
	The loss function is non-negative, $f: \RR^d\rightarrow \RR_{\geq 0}$.
\end{assumption}
\noindent If $\inf\limits_{\bm{w}\in \RR^d} f(\bm{w})\geq -z>-\infty$ for some $z>0$, $f$ can be redefined as $f:=\EE[F(\cdot,\bm{\xi})]+z$ to satisfy Assumption \ref{pos_loss}. Let $B^p_{\epsilon}: \RR^d\rightrightarrows\RR^d$ be the closed $p$-norm ball,  $B^p_{\epsilon}(\bm{w}):=\{\bm{x}\in\RR^d:\|\bm{x}-\bm{w}\|_p\leq\epsilon\}\label{p-closed}$, and in particular let $B^p_{\epsilon}:=B^p_{\epsilon}(\bm{0})\label{p-closed-0}$ for $\epsilon\geq 0$ and $p\geq 1$.

The convergence analysis uses the Clarke $\epsilon$-subdifferential \citep{goldstein1977}$\label{p-n-e-Clarke}$ $\partial^p_{\epsilon} h:\RR^d\rightrightarrows\RR^d$, 
\begin{alignat}{6}
	&&\partial^p_{\epsilon} h(\bm{w}):=\co\{\partial h(\bm{x}): \bm{x}\in B^p_{\epsilon}(\bm{w})\},\nonumber
\end{alignat}
\noindent where $\co$ denotes the convex hull, and $\partial h: \RR^d\rightrightarrows\RR^d\label{Clarke}$ denotes the Clarke subdifferential, which for a locally Lipschitz continuous function $h:\RR^d\rightarrow \RR$ equals
\begin{alignat}{6}
	&&\partial h(\bm{w})=\co\{\bm{v}: \exists \bm{w}^k\rightarrow \bm{w}, \bm{w}^k\in D, \nabla h(\bm{w}^k)\rightarrow \bm{v}\},\label{cldef}
\end{alignat}
where $D$ is the domain of $\nabla h$. The Clarke $\epsilon$-subdifferential is a commonly used relaxation of the Clarke subdifferential for the development and analysis of algorithms for minimizing non-smooth non-convex Lipschitz continuous loss functions. In particular, for any $\epsilon_1,\epsilon_2>0$, algorithms have been developed with 
non-asymptotic convergence guarantees in expectation and with high probability for the approximate stationary point $\dist(\bm{0},\partial^2_{\epsilon_1}f(\bm{w}))\leq \epsilon_2$, see for example \citep{zhang2020,davis2022,metel2022,tian2022}. 

Let $\{\alpha_k\}$ be a positive sequence with $\lim\limits_{k \rightarrow \infty} \alpha_k=0$. The next proposition proves the continuous convergence \citep[Definition 5.41]{rockafellar2009} of the sequence of set-valued mappings $\{\partial^p_{\alpha_k} h\}$ to $\partial h$.
\begin{proposition}
	\label{cont-conv}	
	Let $h:\RR^d\rightarrow \RR$ be a locally Lipschitz continuous function. The sequence of mappings $\{\partial^p_{\alpha_k} h\}$ converges continuously to $\partial h$.
\end{proposition}

\begin{proof}
	The proof uses \citep[Proposition 5.49 (a)]{rockafellar2009} and \citep[Inequality 4(13)]{rockafellar2009}. Consider any $\bm{w}\in\RR^d$ and $\epsilon>0$. For any $\alpha_k>0$, the Pompeiu-Hausdorff distance \citep[Example 4.13]{rockafellar2009} between $\partial h(B^p_{\alpha_k}(\bm{w})):=\{\partial h(\bm{x}): \bm{x}\in B^p_{\alpha_k}(\bm{w})\}$ and $\partial h(\bm{w})$ with respect to the chosen $p$-norm equals
	\begin{alignat}{6}
		&d^p_{\infty}(\partial h(B^p_{\alpha_k}(\bm{w})),\partial h(\bm{w}))\nonumber\\
		=&\inf\{\gamma\geq 0: \partial h(B^p_{\alpha_k}(\bm{w}))\subseteq \partial h(\bm{w})+ B^p_{\gamma}(\bm{w}), \partial h(\bm{w})\subseteq \partial h(B^p_{\alpha_k}(\bm{w}))+ B^p_{\gamma}(\bm{w})\}\nonumber\\
		=&\inf\{\gamma\geq 0: \partial h(B^p_{\alpha_k}(\bm{w}))\subseteq \partial h(\bm{w})+ B^p_{\gamma}(\bm{w})\}.\nonumber
	\end{alignat} 	
	\noindent By the outer semicontinuity of $\partial h$ \citep[Proposition 2.1.5 (d)]{clarke1990}, there exists a $\delta>0$, such that $\partial h(B^p_{\delta}(\bm{w}))\subseteq \partial h(\bm{w})+B^p_{\epsilon}(\bm{w})$, and by the definition of $\{\alpha_k\}$, there exists a $K\in\NN$ such that for $i\geq K$, $\alpha_{i}\leq \frac{\delta}{2}$. For all $\bm{x}\in B^p_{{\alpha_{K}}}(\bm{w})$, 
	$\partial h(B^p_{\alpha_i}(\bm{x}))\subseteq
	\partial h(B^p_{\delta}(\bm{w}))$ by the triangle inequality, hence 
	$\partial h(B^p_{\alpha_i}(\bm{x}))\subseteq \partial h(\bm{w})+B^p_{\epsilon}(\bm{w})$ and 
	$d^p_{\infty}(\partial h(B^p_{\alpha_i}(\bm{x})),\partial h(\bm{w}))\leq \epsilon$. Given that $\partial h(\bm{w})$ and $B^p_{\epsilon}(\bm{w})$ are convex sets, by taking the convex hull of both sides, it also holds for $i\geq K$ and $\bm{x}\in B^p_{{\alpha_{K}}}(\bm{w})$ that   
	$\partial^p_{\alpha_i}h(\bm{x})\subseteq\partial h(\bm{w})+B^p_{\epsilon}(\bm{w})$ \citep[Theorem 1.1.2]{schneider2014}, proving that   
	$\{\partial^p_{\alpha_k} h\}$ converges continuously to $\partial h$.	
\end{proof}

\noindent It is not assumed that $f$ nor $F(\cdot,\bm{\xi})$ are differentiable. We instead define $\widetilde{\nabla} F:\RR^d\times\RR^n\rightarrow \RR^d\label{approx_grad}$ to be a Borel measurable function which equals $\nabla F$ almost everywhere it exists. This can be computed using back propagation for a wide range of neural network architectures made up of elementary functions, see \citep[Proposition 3 \& Theorem 2]{bolte2020} for more details.

In the convergence analysis in Section \ref{perturbedSGD}, iterate perturbation is used with samples of a random variable $\bm{u}:\Omega\rightarrow\RR^d\label{uni-u}$ which is uniformly distributed over $B^{\infty}_{\alpha}$ for an $\alpha>0\label{alpha}$, denoted as $\bm{u}\sim U(B^{\infty}_{\alpha})$. Let $f_{\alpha}:=\EE[f(\cdot+\bm{u})]\label{f-alpha}$ for $\bm{u}\sim U(B^{\infty}_{\alpha})$ be the expected value of the perturbed function $f$. Some useful properties are now listed.
\begin{proposition}{\citep[Propositions 3 \& 6]{metel2022} \&  \citep[Lemma 4.2]{metel2021}}
	\label{old_props}	
	\begin{enumerate}
		\item For any $\bm{w}\in\RR^d$ and $\alpha>0$, with $\bm{u}\sim U(B^{\infty}_{\alpha})$, $\EE[\widetilde{\nabla} F(\bm{w}+\bm{u},\bm{\xi})]=\nabla f_{\alpha}(\bm{w})$ and  
		\item $\nabla f_{\alpha}$ is $L_1^{\alpha}:=\alpha^{-1}\sqrt{d}L_0\label{f-alpha-Lip}$-Lipschitz continuous. 
		\item For almost all $(\bm{w},\bm{\xi})\in\RR^{d+n}$, $\|\widetilde{\nabla} F(\bm{w},\bm{\xi})\|_2\leq L_0(\bm{\xi})$. 
	\end{enumerate}
\end{proposition}

\noindent The following proposition will also be needed, connecting $\nabla f_{\alpha}$ with the $L_{\infty}$-norm Clarke $\alpha$-subdifferential of $f$.
\begin{proposition}
	\label{inclusion}	
	For all $\bm{w}\in\RR^d$ and $\alpha>0$, it holds that 
	$\nabla f_{\alpha}(\bm{w})\in \partial^{\infty}_{\alpha} f(\bm{w})$.
\end{proposition}
\begin{proof}
	Let $\widetilde{\nabla}f$ be a Borel measurable function equal to $\nabla f$ almost everywhere it exists, see \citep[Example A.1]{metel2021} for a method of its construction. It holds that $\nabla f(\bm{w}+\bm{u})\in \partial f(\bm{w}+\bm{u})$ when $f$ is differentiable at $\bm{w}+\bm{u}\in\RR^d$ \citep[Proposition 2.2.2]{clarke1990}, which is for almost all $\bm{u}\in B^{\infty}_{\alpha}$ by Rademacher's theorem. It follows that for almost all $\bm{u}\in B^{\infty}_{\alpha}$, $\widetilde{\nabla} f(\bm{w}+\bm{u})\in  \partial^{\infty}_{\alpha}f(\bm{w})$, hence 
	$\EE[\widetilde{\nabla} f(\bm{w}+\bm{u})]\in \partial^{\infty}_{\alpha}f(\bm{w})$ since $\partial^{\infty}_{\alpha}f(\bm{w})$ is convex and compact \citep[Proposition 2.3]{goldstein1977}. The result holds given that $\nabla f_{\alpha}=\EE[\widetilde{\nabla} f(\cdot+\bm{u})]$ \cite[Proposition 3]{metel2022}.
\end{proof}

\section{Fixed-point Arithmetic Environments}
\label{finiteprec}

Numerical error will be considered in a general form, but to show the applicability of the modelling assumptions, examples are given throughout using fixed-point arithmetic. This is the simplest number format approximating $\RR$, providing a clear view of its induced rounding error, as well as providing non-negligible numerical error for our empirical analysis. Floating-point arithmetic has traditionally been the dominant number format for scientific computing, which in simplified terms, provides an individual scale factor for each number. Motivated by AI model training and inference, much attention has been given to block floating-point arithmetic, where subsets of numbers share the same scale, which can achieve an accuracy close to floating-point with reduced hardware complexity, and energy consumption similar to fixed-point number formats. This number format has been further generalized by the Microscaling specification \cite{mxdoc}, supported by several industry leaders. 

A general fixed-point arithmetic environment is denoted as $\FF\subset \RR\label{fi-e}$ when no further specification is required. For $m,n\in\ZZ_{\geq 0}$, with $m\leq n$, let $[n]_m:=[m,...,n]\label{int-seq}$, and in particular let $[n]:=[n]_1\label{int-seq-1}$. Following \citep{gupta2015}, all $y\in\FF$ are represented in the form of 
\begin{alignat}{6}
	[e_re_{r-1}(...)e_1.d_1d_2(...)d_t],\label{fixp}
\end{alignat}
written in radix complement \citep[Page 1408]{weik2001}, using $r\in\ZZ_{\geq 0}\label{fi-r}$ digits to represent the integer part and $t\in\ZZ_{\geq 0}\label{fi-t}$ digits to represent the fractional part of $y$, with $r+t>0$. Using a base $\beta\in\ZZ_{> 1}\label{base}$, $e_i\in[\beta-1]_0\label{fixed-int}$ for all $i\in [r]$ and $d_i\in[\beta-1]_0\label{fixed-frac}$ for all $i\in[t]$. 

For any $\FF$, let $\Lambda^-$\label{f-s}, $\lambda\label{f-sp}$, and $\Lambda^+\label{f-l}$ denote the smallest, the  smallest positive, and the largest representable numbers, respectively, with its range defined as $\R_{\FF}:=\{x\in\RR: \Lambda^-\leq x\leq \Lambda^+\}\label{range_FF}$. Two forms of rounding will be considered: round to nearest and stochastic rounding. Given an $x\in\R_{\FF}$, let $\lfloor x\rfloor_{\FF}:=\max\{y\in \FF: y\leq x\}\label{x-floor}$ and $\lceil x\rceil_{\FF}:=\min\{y\in \FF: y\geq x\}\label{x-ceil}$, and let $R:\RR\rightarrow \FF\label{x-round}$ denote a function which performs one of the two rounding methods. When rounding an $x\in\R_{\FF}$ using round to nearest,   
\begin{alignat}{6}
	R(x)\in\argmin\limits_{y\in\{\lfloor x\rfloor_{\FF},\lceil x\rceil_{\FF}\}}|y-x|\nonumber. 
\end{alignat}
If $\lceil x\rceil_{\FF}-x=x-\lfloor x\rfloor_{\FF}$, this work does not depend on the use of a specific tie-breaking rule, but we assume that it is deterministic, such as round to even or away \citep[Section 4.3.1]{fp_paper}. For stochastic rounding when $x\notin \FF$,
\begin{alignat}{6}
	R(x):=\begin{cases}
		\lceil x\rceil_{\FF} &\text{with probability }p=\frac{x-\lfloor x\rfloor_{\FF}}{\lceil x\rceil_{\FF}-\lfloor x\rfloor_{\FF}}\\
		\lfloor x\rfloor_{\FF}&\text{with probability }1-p.\\
	\end{cases}\label{sto_ro}
\end{alignat}
Considering the error $\delta:=R(x)-x$, it is well known that $\EE[\delta]=0$, e.g., \cite[Lemma 5.1]{connolly2021}. We also require a bound on its variance. 
\begin{proposition}
	\label{storoundprop}	
	For an $x\in\R_{\FF}$, it holds that 
	\begin{alignat}{6}
		\EE[\delta]=0\quad\text{and}\quad\var(\delta)=\EE[\delta^2]\leq\frac{\beta^{-2t}}{4}.\nonumber
	\end{alignat}
\end{proposition}

\begin{proof}
	Letting $\omega:=\lceil x\rceil_{\FF}-\lfloor x\rfloor_{\FF}$, $\kappa:=x-\lfloor x\rfloor_{\FF}$, and noting that $\lceil x\rceil_{\FF}-x=\omega-\kappa$,  
	\begin{alignat}{6}
		\var[\delta]&=\EE[\delta^2]-\EE[\delta]^2\nonumber\\
		&=(\lceil x\rceil_{\FF} -x)^2\frac{x-\lfloor x\rfloor_{\FF}}{\lceil x\rceil_{\FF}-\lfloor x\rfloor_{\FF}}+(\lfloor x\rfloor_{\FF}-x)^2\left(1-\frac{x-\lfloor x\rfloor_{\FF}}{\lceil x\rceil_{\FF}-\lfloor x\rfloor_{\FF}}\right)\nonumber\\
		&=(\omega-\kappa)^2\frac{\kappa}{\omega}+\kappa^2\frac{\omega-\kappa}{\omega}\nonumber\\	
		&=\frac{\kappa}{\omega}(\omega^2-2\omega\kappa+\kappa^2+\kappa\omega-\kappa^2)\nonumber\\
		&=\kappa\omega-\kappa^2\label{tomax}\\
		&\leq\frac{\omega^2}{2}-\frac{\omega^2}{4}=\frac{(\lceil x\rceil_{\FF}-\lfloor x\rfloor_{\FF})^2}{4}=\frac{\beta^{-2t}}{4},\nonumber
	\end{alignat}
	where the inequality holds given that $\kappa
	=\frac{\omega}{2}$ maximizes the strongly concave function \eqref{tomax}, and the final result holds given that from \eqref{fixp}, $\lceil x\rceil_{\FF}-\lfloor x\rfloor_{\FF}=\beta^{-t}$.	
\end{proof}

\noindent When $x\notin\R_{\FF}$, we assume that $R(x)=\argmin\limits_{y\in\{\Lambda^-,\Lambda^+\}}|y-x|\nonumber$ for both rounding methods, which is similar to how overflows are handled when using round towards zero \citep[Section 7.4]{fp_paper}. 

The basic arithmetic operations $\{+,-,\times,\div\}$ applied to $x,y\in \FF$ using round to nearest give absolute errors bounded by $\{0,0,0.5\beta^{-t},0.5\beta^{-t}\}$, respectively, assuming no overflow \cite[Page 4 \& 5]{wilkinson1965}, 
and when using stochastic rounding, these bounds are increased to $\{0,0,\beta^{-t},\beta^{-t}\}$. Considering now the dot product of two vectors $\bm{x},\bm{y}\in\FF^d$ using stochastic rounding, the rounding error's tail probability can be bounded as follows, where we note that the stochastic rounding errors of each multiplication are independent. 

\section{Past Work on Optimization with Numerical Error}
\label{litreview}
Research on optimization in environments with error is vast when considering stochastic optimization. The minimization of a stochastic function with further numerical error seems to be a topic much less explored. We highlight a few papers which were found to be most relevant to the current research. 

An influential paper for this work was \citep{bertsekas2000}, where the convergence of a gradient method of the form $\bm{w}^{k+1}=\bm{w}^k+\eta^k(\bm{s}^k+\hat{\bm{e}}^k)$ was studied, where $\eta^k$ is a step size, $\bm{s}^k$ is a direction of descent, $\hat{\bm{e}}^k$ is a deterministic or stochastic error, and it is assumed that the loss function $f$ has a Lipschitz continuous gradient. It was proven that $f(\bm{w}^k)$ converges, and if the limit is finite, then $\nabla f(\bm{w}^k)\rightarrow \bm{0}$, without any type of boundedness assumptions. 

In \citep{solodov1998}, a parallel projected incremental algorithm onto a convex compact set is proposed for solving finite-sum problems. It is assumed that there is non-vanishing bounded error when computing subgradients $g\in \partial f_i(\bm{w})$ of each subfunction $f_i$, with a convergence result to an approximate stationary point with an error level relative to the error in computing the subgradients. Each subfunction $f_i$ is assumed to be Lipschitz continuous but regular, i.e., its one-sided directional derivative exists and for all $\bm{v}\in\RR^d$ $f_i'(\bm{w};\bm{v})=\max\limits_{\bm{g}\in\partial f_i(\bm{w})}\langle \bm{g},\bm{v}\rangle$ \citep[Section 2.3]{clarke1990}, which precludes functions with downward cusps such as $\min\{1,\max\{0,1-x\}\}$ (see Example \ref{ex_ramp}). 

Recent work studying the convergence of gradient descent for convex loss functions with a Lipschitz continuous gradient in a low-precision floating-point environment is presented in \citep{xia2022}. 
Biased stochastic rounding schemes are proposed which prevent small gradients from being rounded to zero.  Inequalities are then provided involving the step size, the unit roundoff, and the norms of the gradient and iterates which guarantee either a convergence rate to the optimal solution, or at least the (expected) monotonicity of the loss function values.

The paper \citep{yang2019} studies the algorithm $\bm{w}^{k+1}=R(\bm{w}^k-\eta^k\nabla \tilde{f}(\bm{w}^k))$, where $\nabla \tilde{f}(\bm{w}^k)$ is a stochastic gradient and $R$ performs stochastic rounding into a fixed-point arithmetic environment $\FF$.
It is assumed that the loss function $f$ is strongly convex, with Lipschitz continuous gradient and Hessian, with 
$\nabla \tilde{f}(\bm{w}^k)$ being uniformly bounded from $\nabla f(\bm{w}^k)$ for all $k\in \NN$. Convergence to a neighbourhood of the optimal solution is proven which depends on the precision of $\FF$, 
with an improved dependence proven when considering an exponential moving average of iterates computed in full-precision. 

\section{PISGD with Numerical Error and Adaptive Step Sizes}
\label{perturbedSGD}

The PISGD algorithm with adaptive step sizes is first described with infinite precision. Given an initial iterate $\bm{w}^1\in\RR^d$, we consider a perturbed mini-batch SGD algorithm of the form
\begin{alignat}{6}
	\hspace{-1pt}
	\bm{w}^{k+1}=\bm{w}^k-\frac{\hat{\eta}_k\psi_k}{M}\sum_{i=1}^M\widetilde{\nabla} F(\bm{w}^k+\bm{u}^k,\bm{\xi}^{k,i}),\label{eq:11}
\end{alignat}
where the total step size $\eta_k:=\hat{\eta}_k\psi_k\geq 0\label{ss}$, has a deterministic, $\hat{\eta}_k$, and a stochastic, $\psi_k$, component. The value $M\in\NN\label{MB}$ is the mini-batch size, $\bm{u}^k\sim U(B^{\infty}_{\alpha_k})$ is a sample from a uniform distribution with parameter $\alpha_k>0$, and $\{\bm{\xi}^{k,i}\}$ are $M$ samples of $\bm{\xi}$. In order to model PISGD with numerical error the following notation is introduced:
\begin{enumerate}
	\item $\widehat{\nabla} F:\RR^d\times\RR^n\times\RR^s\rightarrow \RR^d\label{asgrad}$;  $(\bm{w},\bm{\xi},\bm{b})\mapsto\widehat{\nabla} F(\bm{w},\bm{\xi},\bm{b})$ is a Borel measurable function which approximates the stochastic gradient $\widetilde{\nabla} F$, where $\bm{b}\in\RR^s\label{b_sr}$ is a discrete random vector used to perform stochastic rounding,	 
	\item $\hat{\bm{u}}^k\in\RR^d\label{u-approx}$ is an approximation of a sample from the continuous distribution $U(B^{\infty}_{\alpha_k})$, and  
	\item $\hat{\bm{e}}^k\in\RR^d\label{e-error}$ is a random vector which models the error from computing the basic arithmetic operations in \eqref{eq:11}. 
\end{enumerate}

\noindent The proposed model of PISGD with numerical error takes the form 
\begin{alignat}{6}
	\hspace{-1pt}
	\bm{w}^{k+1}=\bm{w}^k-\frac{\hat{\eta}_k\psi_k}{M}\sum_{i=1}^M\widehat{\nabla} F(\bm{w}^k+\hat{\bm{u}}^k,\bm{\xi}^{k,i},\bm{b}^{k,i})+\hat{\bm{e}}^k.\label{eq:1}
\end{alignat}

\noindent The sampling of $\hat{\bm{u}}^k$, $\{\bm{\xi}^{k,i}\}$, and $\{\bm{b}^{k,i}\}$ is assumed to be done independently. Let $\{\F_k\}$ be a filtration on the probability space $(\Omega,\F,P)$, where $\F_k:=\sigma(\hat{\bm{u}}^j,\{\bm{\xi}^{j,i}\},\{\bm{b}^{j,i}\},\psi_j,\hat{\bm{e}}^j: j\in[k])\label{F-filt}$, and let $\{\G_k\}$ be a sequence of $\sigma$-algebras, where $\G_k:=\sigma(\hat{\bm{u}}^k,\{\bm{\xi}^{k,i}\},\{\bm{b}^{k,i}\},\psi_k)\label{G-sig}$.
The $\sigma$-algebra $\G_k$ is used to analyze the error $\hat{\bm{e}}^k\in\RR^d$. The algorithm step \eqref{eq:1} can be broken down into two half steps, where at step ``$k+\frac{1}{2}$", all elements of $S^k:=\{\bm{w}^k,\hat{\eta}_k,\psi_k,M,\{\widehat{\nabla} F(\bm{w}^k+\hat{\bm{u}}^k,\bm{\xi}^{k,i},\bm{b}^{k,i})\}\}\label{Sk}$ have been computed, after which $w^{k+1}$ is computed with numerical error $\hat{\bm{e}}^k$ using the elements of $S^k$.  The iterate $\bm{w}^k$ is $\F_{k-1}$-measurable, and all elements within $S^k$ are $\sigma(\F_{k-1},\G_k)$-measurable.

\subsection{Modelling Details of PISGD with Numerical Error and Adaptive Step Sizes}

\subsubsection{Description of $\hat{\bm{u}}^k$}
The original $\bm{u}^k\sim U(B^{\infty}_{\alpha_k})$ is replaced by a sample $\hat{\bm{u}}^k\in\RR^d$ from a probability distribution $\widehat{P}^k\label{P-u}$, where the sequence of probability distributions $\{\widehat{P}^k\}$ and parameters $\{\alpha_k\}$ are assumed to be deterministic. This allows for modelling the approximate sampling of $U(B^{\infty}_{\alpha_k})$ using finite precision, such as through discretization.

\subsubsection{Description of $\bm{b}^{k,i}$}
The inclusion of the random vector $\bm{b}\in\RR^s$ in $\widehat{\nabla} F$ models the use of stochastic rounding. The size $s\in\NN$ of $\bm{b}$ is equal to the number of 
rounding operations required to approximately compute $\widetilde{\nabla} F$, see \cite[Section 7]{croci2022} for an overview of the implementation of stochastic rounding in practice, which generally consists of adding random bits and truncating the result. Another approach sufficient for our model is to sample from a discretized version $b_j$ of $\hat{b}_j\sim U([0,1])$ and round up if $b_j\leq p$ or down otherwise for all $j\in[s]$ following \eqref{sto_ro}. It is assumed that for all $k\in\NN$ and $j\in[s]$, $b^k_j\in\RR$ is a discrete uniformly distributed random variable over a finite set $V^k_j\subset\RR$. The distribution of $\bm{b}^k$ is denoted as $U(V^k)$, where $V^k:=\{\bm{b}: \PP(\bm{b}^k=\bm{b})>0\}\label{V-b}$ is the support of  $\bm{b}^k$. In \eqref{eq:1}, the set $\{\bm{b}^{k,i}\}\subset\RR^s$ contains M samples of $\bm{b}^k\sim U(V^k)$. This model supports the use of random bits as in practice, or a discretization of $[0,1]$. The support $V^k$ is allowed to change through time to adjust the precision of the stochastic rounding implementation. 

\subsubsection{Modelling the Error of $\widehat{\nabla} F$}

The required accuracy of the perturbed approximate stochastic gradient $\widehat{\nabla} F$ is contained in the following assumption. 

\begin{assumption}\label{approx_assump}
	There exist constants $c_1>0, c_2>0$, and a $K\in\NN$ such that for all $k\geq K$,   
	\begin{alignat}{6}
		\langle \EE[\widehat{\nabla} F(\bm{w}^k+\hat{\bm{u}}^k,\bm{\xi},\bm{b}^k)|\F_{k-1}],\nabla f_{\alpha_k}(\bm{w}^k)\rangle&\geq c_1 \|\nabla f_{\alpha_k}(\bm{w}^k)\|^2_2\quad\text{and} \label{errorineq}\\	
		\EE[\|\widehat{\nabla} F(\bm{w}^k+\hat{\bm{u}}^k,\bm{\xi},\bm{b}^k)\|^2_2|\F_{k-1}]&\leq c_2Q\label{errorupper}
	\end{alignat}
	almost surely, where $\hat{\bm{u}}^k\sim \widehat{P}^k$, $\bm{b}^k\sim U(V^k)$,  $f_{\alpha_k}:=\EE[f(\cdot+\bm{u}^k)]$ for $\bm{u}^k\sim U(B^{\infty}_{\alpha_k})$, and recalling that $Q:=\EE[L_0(\bm{\xi})^2]$.	
\end{assumption}

\noindent Inequalities \eqref{errorineq} and \eqref{errorupper} are variants of classic error assumptions, see \cite[Equations (4.3) \& (4.4)]{levitin1966}, \cite[Equation (1.5)]{bertsekas2000}, and \cite[Equation 4.7]{bottou2018}, tailored to our problem setting. Inequality \eqref{errorineq} states that the conditional expectation of $-\widehat{\nabla} F(\bm{w}^k+\hat{\bm{u}}^k,\bm{\xi},\bm{b}^k)$ must be a direction of descent for $f_{\alpha_k}$ at $\bm{w}^k$ almost surely when $k\in\NN$ is sufficiently large. When $\widehat{\nabla} F(\bm{w},\bm{\xi},\bm{b}^k)=\widetilde{\nabla} F(\bm{w},\bm{\xi})$ for almost all $(\bm{w},\bm{\xi})\in B^{\infty}_{\alpha_k}(\bm{w}^k)\times\RR^n$ and all $\bm{b}^k\in V^k$, and $\hat{\bm{u}}^k\sim U(B^{\infty}_{\alpha_k})$, inequalities \eqref{errorineq} and \eqref{errorupper} are satisfied with $c_1=c_2=1$ from Propositions \ref{old_props}(1) and \ref{old_props}(3). Given that any $0<c_1<1$ and $1<c_2<\infty$ are valid, Assumption \ref{approx_assump} allows the random variable $\widehat{\nabla} F(\bm{w}^k+\hat{\bm{u}}^k,\bm{\xi},\bm{b}^k)$ to be an approximation of $\widetilde{\nabla} F(\bm{w}^k+\bm{u}^k,\bm{\xi})$ with nontrivial error. Even when stochastic rounding is used, it cannot be assumed that the rounding error is unbiased with $c_1=1$. In particular, this negative result holds for the ResNet models \cite{he2016} used in the experiments in Section \ref{empirical}, which use batch normalization \citep{ioffe2015}.	
\begin{proposition}
\label{bnerror}
The expected rounding error from computing batch normalization and its gradient using stochastic rounding is in general non-zero.
\end{proposition} 

\begin{proof}
Using the notation of the definition of batch normalization written in \citep[Algorithm 1]{ioffe2015}, consider a mini-batch of size 2, with $x_1=2$, $x_2=1$, $\epsilon=0.25$, $\gamma=1$, and $\beta=0$, and an $\FF$ with $\hat{\FF}:=\{-0.5,0.25,0.5,1,1.5,2,3\}\subseteq\FF$, e.g., base $2$ with $r\geq 3$ and $t\geq 2$. The values $v_i:=x_i-\mu_{\beta}$ for $i\in[2]$ and $z:=\sigma^2_{\beta}+\epsilon$ can be computed exactly with $v_1=z=0.5$. The output for $x_1$ can be written as $y_1=\frac{v_1}{\sqrt{z}+\delta_1}+\delta_2$, where $\delta_1$ is the stochastic rounding error from the square root operation and $\delta_2$ is the subsequent rounding error from the division, with   	
	\begin{alignat}{6}		
		&\EE[y_1]&=&\EE[\EE[\frac{v_1}{\sqrt{z}+\delta_1}+\delta_2|v_1,z,\delta_1]]\nonumber\\
		&&=&\EE[\frac{v_1}{\sqrt{z}+\delta_1}+\EE[\delta_2|v_1,z,\delta_1]]\nonumber\\
		&&=&\EE[\frac{v_1}{\sqrt{z}+\delta_1}]\nonumber\\
		&&=&\frac{v_1}{\lceil\sqrt{z}\rceil_{\FF}}\frac{\sqrt{z}-\lfloor \sqrt{z}\rfloor_{\FF}}{\lceil \sqrt{z}\rceil_{\FF}-\lfloor \sqrt{z}\rfloor_{\FF}}+\frac{v_1}{\lfloor\sqrt{z}\rfloor_{\FF}}\left(1-\frac{\sqrt{z}-\lfloor \sqrt{z}\rfloor_{\FF}}{\lceil \sqrt{z}\rceil_{\FF}-\lfloor \sqrt{z}\rfloor_{\FF}}\right)\nonumber\\		
		&&=&\frac{v_1}{\lceil \sqrt{z}\rceil_{\FF}-\lfloor \sqrt{z}\rfloor_{\FF}}
		\left(\frac{\sqrt{z}-\lfloor \sqrt{z}\rfloor_{\FF}}{\lceil \sqrt{z}\rceil_{\FF}}+\frac{\lceil \sqrt{z}\rceil_{\FF}-\sqrt{z}}{\lfloor \sqrt{z}\rfloor_{\FF}}\right).\nonumber
	\end{alignat}
	Assume that the expected rounding error is zero:
	\begin{alignat}{6}	
		&&&\frac{v_1}{\lceil \sqrt{z}\rceil_{\FF}-\lfloor \sqrt{z}\rfloor_{\FF}}
		\left(\frac{\sqrt{z}-\lfloor \sqrt{z}\rfloor_{\FF}}{\lceil \sqrt{z}\rceil_{\FF}}+\frac{\lceil \sqrt{z}\rceil_{\FF}-\sqrt{z}}{\lfloor \sqrt{z}\rfloor_{\FF}}\right)=\frac{v_1}{\sqrt{z}}\nonumber\\
		&&\Rightarrow&\sqrt{z}
		\left(\frac{\sqrt{z}-\lfloor \sqrt{z}\rfloor_{\FF}}{\lceil \sqrt{z}\rceil_{\FF}}+\frac{\lceil \sqrt{z}\rceil_{\FF}-\sqrt{z}}{\lfloor \sqrt{z}\rfloor_{\FF}}\right)=\lceil \sqrt{z}\rceil_{\FF}-\lfloor \sqrt{z}\rfloor_{\FF}\nonumber\\
		&&\Rightarrow&
		\sqrt{z}\left(\frac{\lceil \sqrt{z}\rceil_{\FF}}{\lfloor \sqrt{z}\rfloor_{\FF}}-
		\frac{\lfloor \sqrt{z}\rfloor_{\FF}}{\lceil \sqrt{z}\rceil_{\FF}}\right)=\lceil \sqrt{z}\rceil_{\FF}-\lfloor \sqrt{z}\rfloor_{\FF}+z\left(\frac{1}{\lfloor \sqrt{z}\rfloor_{\FF}}-\frac{1}{\lceil \sqrt{z}\rceil_{\FF}}\right).\nonumber
	\end{alignat}	
	For any $\FF$ with $\hat{\FF}\subseteq\FF$ this cannot hold given that $\sqrt{z}$ is irrational and $\lceil \sqrt{z}\rceil_{\FF}>\lfloor \sqrt{z}\rfloor_{\FF}>0$: The left-hand side is an irrational number, whereas the right-hand side is rational. Batch normalization suffers from biased rounding error due to the division by $\sqrt{z}$, which can also be found when computing its gradient \cite[Section 3]{ioffe2015}.
\end{proof}

\noindent For simplicity let 
$\widehat{\nabla} F^{k,i}(\bm{w}^k):=\widehat{\nabla} F(\bm{w}^k+\hat{\bm{u}}^k,\bm{\xi}^{k,i},\bm{b}^{k,i})\label{F-tilde-sim}$ for $i\in[M]$, and
$\widehat{\nabla} \overline{F}^k(\bm{w}^k):=\frac{1}{M}\sum_{i=1}^M\widehat{\nabla} F^{k,i}(\bm{w}^k)\label{F-tilde-sim2}$. 
\begin{proposition}
	\label{normbound}
	For all $k\geq K$ from Assumption \ref{approx_assump}, $\EE[\|\widehat{\nabla} \overline{F}^k(\bm{w}^k)\|^2_2|\F_{k-1}]\leq c_2Q$ almost surely.
\end{proposition}
\begin{proof}
	\begin{alignat}{6}
		&\EE[\|\widehat{\nabla} \overline{F}^k(\bm{w}^k)\|^2_2|\F_{k-1}]&=&\EE\left[\sum_{j=1}^d\left(\frac{1}{M}\sum_{i=1}^M\widehat{\nabla}F^{k,i}_j(\bm{w}^k)\right)^2\middle|\F_{k-1}\right]\nonumber\\
		&&\leq&\EE\left[\sum_{j=1}^d\frac{1}{M}\sum_{i=1}^M\widehat{\nabla}F^{k,i}_j(\bm{w}^k)^2\middle|\F_{k-1}\right]\nonumber\\
		&&=&\frac{1}{M}\sum_{i=1}^M\EE[\|\widehat{\nabla}F^{k,i}(\bm{w}^k)\|^2_2|\F_{k-1}]\stackrel{\text{a.s.}}{\leq}c_2Q,\nonumber		
	\end{alignat}
	where the first inequality uses Jensen's inequality and the second uses \eqref{errorupper}. 
\end{proof}

\subsubsection{Discussion on Modelling Assumptions}

The use of stochastic rounding and iterate perturbation when computing $\widehat{\nabla} F$ has been modelled for completeness, but in terms of convergence analysis, all that is needed is a function $\widehat{\nabla} F(w^k)$, ignoring all other arguments, for which Assumption \ref{approx_assump} holds. 

There is generally a large gap between the observed rounding error and what can be guaranteed theoretically. For round to nearest using floating-point arithmetic, ``the constants (in an error bound) usually cause the bound to overestimate the actual error by orders of magnitude'' \cite[pg. 65]{higham2002}. For the dot product of two vectors $\bm{x}$, $\bm{y}\in\GG^n$, where $\GG$ denotes a floating-point arithmetic environment, the absolute error is bounded by $\gamma|\bm{x}|^T|\bm{y}|$, for $\gamma:=\frac{nu}{1-nu}$, where $u$ is the unit roundoff \cite[Eq. (3.5)]{higham2002}. Considering the number formats used in modern GPUs for machine learning training \cite{sevegnani2025}, namely FP16 ($u=2^{-11}$), BF16 ($u=2^{-8}$), FP8 E4M3 ($u=2^{-4}$), and FP8 E5M2 ($u=2^{-3}$), and that this bound requires $nu<1$, it fails to hold when $n\geq 2048, 256, 16,$ and $8$, respectively. Using stochastic rounding, the absolute error of dot products can be guaranteed to hold with probability at least $T(\lambda,n):=1-2n\exp(-0.5\lambda^2)$ with $\gamma=\exp((\lambda\sqrt{n}u+nu^2)(1-u)^{-1})-1$ \cite[Theorem 4.8]{connolly2021}. For the ResNet models considered in Section \ref{empirical}, a single forward pass requires up to 71.48 million FLOPs \cite[Table 1]{guan2018}. Using the approximation that back propagation requires twice as many operations as forward propagation \cite[Appendix C.1]{zhou2021} results in 214.44 million rounding operations per gradient calculation. Considering a dot product with that many FLOPs (multiply-adds), choosing $\lambda=6.413$, which only gives a probability bound $T(\lambda,n)<0.5$, results in $\gamma>1.377E42$ when using FP16. Consider now the function $\frac{1}{2}\bm{w}^T\bm{A}\bm{w}$, where $\bm{A}$ is symmetric, $\bm{w}\in\text{BF16}^{14,634}$, and again $\lambda=6.413$. Computing the gradient, $\bm{A}\bm{w}$, requiring $14,634^2<214.44$ million FLOPs, results in a per element $\gamma>25.21$ (an absolute error bound $>25.21 |\bm{A}_i\|\bm{w}|$), with again $T(\lambda,n)<0.5$ \cite[Theorem 4.9]{connolly2021}. 

From these simple examples, trying to bound the rounding error of deep learning models, besides being complicated given the large number of layers and nonlinear functions employed, is likely to result in a bound of little use. For this reason, it is perhaps more practical to view $\widehat{\nabla} F$ as a black-box function when considering its rounding error, and to rely on an empirical verification of Assumption \ref{approx_assump} as needed. An example of how this can be done is given in Section \ref{testing_assum}. 

At the same time, it is important to show that Assumption \ref{approx_assump} can be satisfied using fixed-point arithmetic, which is done in the following proposition, where explicit values for $c_1$ and $c_2$ are given. The proposition considers binary classification using a linear predictor and the ramp loss \cite[Section 15.2.3]{shalev2014}, which is a non-convex Lipschitz continuous function which is not Clarke regular. 

\begin{proposition}
\label{ex_ramp}
Let $\bm{\xi}=[\bm{x}^T,y]^T$, where $\bm{x}\in\RR^d$ and $y\in\{-1,1\}$, with $F(\bm{w},\bm{\xi}):=\min\{1,\max\{0,1-y\langle\bm{x},\bm{w}\rangle\}\}$. Assume that $\bm{x}$ is a normalized random variable, $\bm{x}=\frac{\bm{x}'}{\|\bm{x}'\|_2}$, where $\bm{x}'\sim N(\bm{0},\bm{I})$, and that $y\in\{-1,1\}$ is a random variable independent of $\bm{x}$. 
Assume that $0.1\leq\|\bm{w}\|_2\leq 0.9$, $d\leq100$, $\beta\geq2$ and $t\geq14$ for $\FF$, and that $\alpha\leq0.001$. By setting $\widehat{\nabla} F(\bm{w},\bm{\xi},\bm{b})$, with $\hat{\bm{u}}=0$ almost surely, as an element of $\partial F(\bm{w},\bm{\xi})$ with rounding error, Assumption \ref{approx_assump} holds with $c_1=0.49$ and $c_2=1.01$.
\end{proposition}
\begin{proof}
From the assumptions, $\bm{x}\sim U(S^{d-1})$, where $S^{d-1}:=\{\bm{z}\in\RR^d:\|\bm{z}\|_2=1\}$ \cite[Theorem 1.5.6]{muirhead1982}, and $\bm{z}:=y\bm{x}\sim U(S^{d-1})$ as well, given the orthogonal invariance of $\bm{x}$ and the independence of $y$ and $\bm{x}$. The Lipschitz constant $L_0(\bm{\xi})=\|\bm{z}\|_2=1$: 
\begin{alignat}{6}		
    &&&|F(\bm{w},\bm{\xi})-F(\bm{w}',\bm{\xi})|\nonumber\\
    &&=&|\min\{1,\max\{0,1-\langle\bm{z},\bm{w}\rangle\}\}-\min\{1,\max\{0,1-\langle\bm{z},\bm{w}'\rangle\}\}|\nonumber\\
    &&\leq&|\langle \bm{z},\bm{w}-\bm{w}'\rangle|\nonumber\\
    &&\leq&\|\bm{z}\|_2\|\bm{w}-\bm{w}'\|_2,\nonumber	
\end{alignat}
where the first inequality uses the nonexpansiveness of the projection $\min\{1,\max\{0,\cdot\}\}$ onto $[0,1]$ \cite[Theorem 5.4(b)]{beck2017}.

	Using the definition \eqref{cldef} for the non-differentiable points,	
	\begin{alignat}{6}		
		\partial F(\bm{w},\bm{\xi})=
		\begin{cases}
            \bm{0}&{\text{if }} \langle\bm{z},\bm{w}\rangle<0\nonumber\\
			-\chi_1 \bm{z}\quad\chi_1\in[0,1]&{\text{if }} \langle\bm{z},\bm{w}\rangle=0\nonumber\\			
			-\bm{z}&{\text{if }} 0<\langle\bm{z},\bm{w}\rangle<1\nonumber\\
			-\chi_2 \bm{z}\quad\chi_2\in[0,1]&{\text{if }} \langle\bm{z},\bm{w}\rangle=1\nonumber\\
        \bm{0}&{\text{if }} \langle\bm{z},\bm{w}\rangle>1.\nonumber
		\end{cases}	
	\end{alignat}	
	
	\noindent The approximate gradient $\widetilde{\nabla} F(\bm{w},\bm{\xi})$ is set to the element of $\partial F(\bm{w},\bm{\xi})$ with $\chi_1=\chi_2=0$, which will always be evaluated at $\bm{w}+\bm{u}$ for $\bm{u}\sim U(B^{\infty}_{\alpha})$:
	\begin{alignat}{6}		
		\widetilde{\nabla} F(\bm{w}+\bm{u},\bm{\xi})=
		\begin{cases}
            \bm{0}&{\text{if }} \langle\bm{z},\bm{w}+\bm{u}\rangle\leq 0\nonumber\\						
			-\bm{z}&{\text{if }} 0<\langle\bm{z},\bm{w}+\bm{u}\rangle<1\nonumber\\			
            \bm{0}&{\text{if }} \langle\bm{z},\bm{w}+\bm{u}\rangle\geq1,\nonumber
		\end{cases}	
	\end{alignat}
    and using Proposition \ref{old_props}(1), $\nabla f_{\alpha}(\bm{w})=\EE[\widetilde{\nabla} F(\bm{w}+\bm{u},\bm{\xi})]$.   
 
    The computation of $\langle\bm{z},\bm{w}\rangle$ in finite precision is modelled as 	
	$$\left(\langle\bm{z},\bm{w}\rangle+\sum_{j=1}^d \delta^{zw}_j\right)=\langle\bm{z},\bm{w}\rangle+\delta^{zw},$$	
	where $\delta^{zw}_j=R(z_jw_j)-z_jw_j$ and $\delta^{zw}:=\sum_{j=1}^d \delta^{zw}_j$. The gradient $-\bm{z}$ with rounding error is modelled as $-(\bm{z}+\bm{\delta}^{z})$, where $\delta^{z}_j=R(z_j)-z_j$. Multiplying by $y\in\{-1,1\}$ does not incur any rounding error, so analyzing these operations with $y\bm{x}$ instead of the fused value $\bm{z}$ does not change the rounding error. Let 	
	\begin{alignat}{6}		
		&\widehat{\nabla} F(\bm{w},\bm{\xi},\bm{b})
		=&\begin{cases}
            \bm{0}&{\text{if }} \langle\bm{z},\bm{w}\rangle+\delta^{zw}\leq0\nonumber\\					
			-\bm{z}-\bm{\delta}^{z}&{\text{if }} 0<\langle\bm{z},\bm{w}\rangle+\delta^{zw}<1\nonumber\\			
			\bm{0}&{\text{if }} \langle\bm{z},\bm{w}\rangle+\delta^{zw}\geq 1,\nonumber
		\end{cases}	
	\end{alignat}	
    \noindent where $\widehat{P}$ has been chosen as a degenerate probability distribution with $\PP(\hat{\bm{u}}=\bm{0})=1$, with $\hat{\bm{u}}$ omitted. 
    
    Given that $|\langle\bm{z},\bm{w}+\bm{u}\rangle|\leq \|w\|_2+\alpha\sqrt{d}\leq 0.9+0.001*10<1$, and $|\delta^{zw}|\leq \sum_{j=1}^d |\delta^{zw}_j|\leq d\beta^{-t}\leq 100*2^{-14}<0.1$, the events $\langle\bm{z},\bm{w}+\bm{u}\rangle\geq1$ and $ \langle\bm{z},\bm{w}\rangle+\delta^{zw}\geq 1$ in the definitions of $\widetilde{\nabla} F(\bm{w}+\bm{u},\bm{\xi})$ and $\widehat{\nabla} F(\bm{w},\bm{\xi},\bm{b})$ do not occur almost surely.     
	The functions $\ind_{\widetilde{F}}:=\ind_{0<\langle\bm{z},\bm{w}+\bm{u}\rangle}$ and $\ind_{\widehat{F}}:=\ind_{0<\langle\bm{z},\bm{w}\rangle+\delta^{zw}}$ therefore almost surely indicate where $\widetilde{\nabla} F(\bm{w}+\bm{u},\bm{\xi})$ and $\widehat{\nabla} F(\bm{w},\bm{\xi},\bm{b})$ have non-zero gradients, respectively. It follows that   
	\begin{alignat}{6}	
		\EE[\widetilde{\nabla} F(\bm{w}+\bm{u},\bm{\xi})]&=\EE[-\bm{z}\ind_{\widetilde{F}}]\nonumber\\
        &=\EE_{\bm{u}}[\EE_{\bm{z}}[-\bm{z}\ind_{\widetilde{F}}]].\label{b_exp2}
	\end{alignat}	
Concentrating on the inner expectation, 
	\begin{alignat}{6}	
	\EE_{\bm{z}}[-\bm{z}\ind_{\widetilde{F}}]=\EE[-z_1\ind_{z_1>0}]\frac{\bm{w}+\bm{u}}{\|\bm{w}+\bm{u}\|_2}\label{rotational_inv}
	\end{alignat}	
by the rotational invariance of symmetrically distributed $\bm{z}$. Using the marginal distribution of $z_1$, $f(z_1)=\frac{\Gamma(\frac{d}{2})}{\sqrt{\pi}\Gamma(\frac{d-1}{2})}(1-z_1^2)^{\frac{d-3}{2}}$ \cite[Problem 1.32 (a)]{muirhead1982},
	
	\begin{alignat}{6}	
		\EE[z_1\ind_{z_1>0}]=&\frac{\Gamma(\frac{d}{2})}{\sqrt{\pi}\Gamma(\frac{d-1}{2})}\int_{0}^1z_1(1-z_1^2)^{\frac{d-3}{2}}dz_1=\frac{\Gamma(\frac{d}{2})}{\sqrt{\pi}\Gamma(\frac{d-1}{2})(d-1)}.\nonumber	
	\end{alignat}
Applying the double inequality $(\frac{v}{v+s})^{1-s}\leq \frac{\Gamma(v+s)}{v^s\Gamma(v)}\leq 1$ \cite[Eq. 7]{wendel1948} for $v>0$ and $0<s<1$, it holds that 
	\begin{alignat}{6}	
		\frac{1}{\sqrt{2\pi d}}\leq\EE[z_1\ind_{z_1>0}]\leq\frac{1}{\sqrt{2\pi(d-1)}},\nonumber
	\end{alignat}	
	\noindent hence $\EE_{\bm{z}}[-\bm{z}\ind_{\widetilde{F}}]=-\tau\frac{\bm{w}+\bm{u}}{\|\bm{w}+\bm{u}\|_2}$, where $\tau\in\left[\frac{1}{\sqrt{2\pi d}},\frac{1}{\sqrt{2\pi(d-1)}}\right]$. Continuing from \eqref{b_exp2},
	\begin{alignat}{6}	
		\EE[\widetilde{\nabla} F(\bm{w}+\bm{u},\bm{\xi})]&=-\tau\EE\left[\frac{\bm{w}+\bm{u}}{\|\bm{w}+\bm{u}\|_2}\right] \label{b_exp3}.
	\end{alignat}
    When $\alpha=0$, \eqref{b_exp3} simplifies to 
    $\widetilde{\nabla} F(\bm{w},\bm{\xi})=-\tau\frac{\bm{w}}{\|\bm{w}\|_2}$. Bounding the error from perturbation,
	\begin{alignat}{6}	
		\left\|\EE[\widetilde{\nabla} F(\bm{w}+\bm{u},\bm{\xi})]+\tau\frac{\bm{w}}{\|\bm{w}\|_2}\right\|_2&\leq\tau\EE\left[\frac{\|\|\bm{w}\|_2(\bm{w}+\bm{u})-\|\bm{w}+\bm{u}\|_2\bm{w}\|_2}{\|\bm{w}+\bm{u}\|_2\|\bm{w}\|_2}\right]\nonumber\\
        &=\tau\EE\left[\frac{\|(\|\bm{w}\|_2-\|\bm{w}+\bm{u}\|_2)\bm{w}+\|\bm{w}\|_2\bm{u}\|_2}{\|\bm{w}+\bm{u}\|_2\|\bm{w}\|_2}\right]\nonumber\\
        &\leq\tau\EE\left[\frac{2\|\bm{u}\|_2\|\bm{w}\|_2}{\|\bm{w}+\bm{u}\|_2\|\bm{w}\|_2}\right]\nonumber\\
        &\leq2\tau\EE\left[\frac{\|\bm{u}\|_2}{\|\bm{w}\|_2-\|\bm{u}\|_2}\right]\nonumber\\      
        &\leq\frac{2\tau\EE[\|\bm{u}\|_2]}{\|\bm{w}\|_2-\alpha\sqrt{d}}\nonumber\\      
        &\leq\frac{2\tau\alpha\sqrt{d}}{\sqrt{3}(\|\bm{w}\|_2-\alpha\sqrt{d})},\label{pert_bound}  
	\end{alignat}        
	where in particular,
	\begin{alignat}{6}	     
		\frac{\alpha\sqrt{d}}{(\|\bm{w}\|_2-\alpha\sqrt{d})}&\leq0.\bar{1}=:\widetilde{\rho}\nonumber
	\end{alignat}   	
	by maximizing $\alpha$ and $d$, and minimizing $\|\bm{w}\|_2$, with $\alpha=0.001$, $d=100$, and $\|\bm{w}\|_2=0.1$.
	
	To analyze $\EE[\widehat{\nabla} F(\bm{w},\bm{\xi},\bm{b})]$, let $\ind_F:=\ind_{0<\langle\bm{z},\bm{w}\rangle}$, i.e., $\ind_{\widehat{F}}$ without rounding error,
	\begin{alignat}{6}	
		\EE[\widehat{\nabla} F(\bm{w},\bm{\xi},\bm{b})]&=\EE[-(\bm{z}+\bm{\delta}^{z})\ind_{\widehat{F}}]\nonumber\\
        &=\EE[\EE[-(\bm{z}+\bm{\delta}^{z})\ind_{\widehat{F}}|\bm{z},\delta^{zw}]]\nonumber\\
        &=\EE[-\bm{z}\ind_{\widehat{F}}]\nonumber\\
        &=\EE[-\bm{z}(\ind_{\widehat{F}}+\ind_F-\ind_F)]\nonumber\\
        &=-\tau\frac{\bm{w}}{\|\bm{w}\|_2}+\EE[-\bm{z}(\ind_{\widehat{F}}-\ind_F)].\nonumber
	\end{alignat}	

    \noindent Given that $\delta^{zw}\leq d\beta^{-t}$, when $|\langle \bm{z},\bm{w}\rangle|>d\beta^{-t}$, it holds that $\ind_{\widehat{F}}=\ind_F$. In order to bound the error from stochastic rounding, let $\ind_{F^{d\beta^{-t}}}:=\ind_{0<\langle\bm{z},\bm{w}\rangle\leq d\beta^{-t}}$ and $\ind_{F^{-d\beta^{-t}}}:=\ind_{-d\beta^{-t}\leq\langle\bm{z},\bm{w}\rangle\leq0}$,      
    \begin{alignat}{6}	
    \left\|\EE[\widehat{\nabla} F(\bm{w},\bm{\xi},\bm{b})]+\tau\frac{\bm{w}}{\|\bm{w}\|_2}\right\|_2&=\left\|\EE[-\bm{z}(\ind_{\widehat{F}}-\ind_F)]\right\|_2\nonumber\\
    &=\left\|\EE[-\bm{z}(\ind_{\widehat{F}}-\ind_F)\ind_{F^{d\beta^{-t}}}]+\EE[-\bm{z}(\ind_{\widehat{F}}-\ind_F)\ind_{F^{-d\beta^{-t}}}]\right\|_2\nonumber\\
    &=\left\|\EE[-\bm{z}(\ind_{\widehat{F}}-1)\ind_{F^{d\beta^{-t}}}]+\EE[-\bm{z}\ind_{\widehat{F}}\ind_{F^{-d\beta^{-t}}}]\right\|_2\nonumber\\
    &\leq\left\|\EE[-\bm{z}(\ind_{\widehat{F}}-1)\ind_{F^{d\beta^{-t}}}]\|_2+\|\EE[-\bm{z}\ind_{\widehat{F}}\ind_{F^{-d\beta^{-t}}}]\right\|_2\nonumber\\
    &\leq\left\|\EE[\bm{z}\ind_{F^{d\beta^{-t}}}]\|_2+\|\EE[-\bm{z}\ind_{F^{-d\beta^{-t}}}]\right\|_2\nonumber\\
    &=2\left\|\EE[\bm{z}\ind_{F^{d\beta^{-t}}}]\right\|_2,\nonumber
	\end{alignat}	
	\noindent where the final equality holds given that $\bm{z}=-\bm{z}$ in distribution. Concentrating on the inner expectation, which is similar to \eqref{rotational_inv},  $\EE[\bm{z}\ind_{F^{d\beta^{-t}}}]=\EE\left[\bm{z}\ind_{0<\langle\bm{z},\frac{\bm{w}}{\|\bm{w}\|_2}\rangle\leq\frac{d\beta^{-t}}{\|\bm{w}\|_2}}\right]=\EE\left[z_1\ind_{0<z_1\leq\frac{d\beta^{-t}}{\|\bm{w}\|_2}}\right]\frac{w}{\|\bm{w}\|_2}$, hence	
	\begin{alignat}{6}	
		\EE[\bm{z}\ind_{F^{d\beta^{-t}}}]=&\frac{\Gamma(\frac{d}{2})}{\sqrt{\pi}\Gamma(\frac{d-1}{2})}\int_{0}^{\frac{d\beta^{-t}}{\|\bm{w}\|_2}}z_1(1-z_1^2)^{\frac{d-3}{2}}dz_1\frac{w}{\|\bm{w}\|_2}\nonumber\\
                                            =&\tau\left(1-\left(1-\frac{d^2\beta^{-2t}}{\|\bm{w}\|^2_2}\right)^{\frac{d-1}{2}}\right)\frac{w}{\|\bm{w}\|_2}\text{, and}\nonumber	
	\end{alignat}    
    \begin{alignat}{6}	
    \left\|\EE[\widehat{\nabla} F(\bm{w},\bm{\xi},\bm{b})]+\tau\frac{\bm{w}}{\|\bm{w}\|_2}\right\|_2&\leq 2\tau\left(1-\left(1-\frac{d^2\beta^{-2t}}{\|\bm{w}\|^2_2}\right)^{\frac{d-1}{2}}\right).\label{round_bound}
	\end{alignat}
	\noindent In addition, 
	\begin{alignat}{6}	     
		1-\left(1-\frac{d^2\beta^{-2t}}{\|\bm{w}\|^2_2}\right)^{\frac{d-1}{2}}&<0.169=:\hat{\rho}\nonumber
	\end{alignat}   	
	by maximizing $d$ and minimizing $\|\bm{w}\|_2$, $\beta$, and $t$, with $d=100$, $\|\bm{w}\|_2=0.1$, $\beta=2$, and $t=14$.	
	
	Given the error bounds from perturbation \eqref{pert_bound} and rounding error \eqref{round_bound}, a $c_1>0$ in \eqref{errorineq} can now be derived: Let 
	$\widetilde{\Delta}:=\EE[\widetilde{\nabla} F(\bm{w}+\bm{u},\bm{\xi})]+\tau\frac{\bm{w}}{\|\bm{w}\|_2}$ and 
	$\widehat{\Delta}:=\EE[\widehat{\nabla} F(\bm{w},\bm{\xi},\bm{b})]+\tau\frac{\bm{w}}{\|\bm{w}\|_2}$,    	
    \begin{alignat}{6}
	&\langle\EE[\widehat{\nabla} F(\bm{w},\bm{\xi},\bm{b})],\nabla f_{\alpha}(\bm{w})\rangle\nonumber\\
	=&\langle \EE[\widehat{\nabla} F(\bm{w},\bm{\xi},\bm{b})],\EE[\widetilde{\nabla} F(\bm{w}+\bm{u},\bm{\xi})]\rangle\nonumber\\
	=&\left\langle -\tau\frac{\bm{w}}{\|\bm{w}\|_2} +\widehat{\Delta},-\tau\frac{\bm{w}}{\|\bm{w}\|_2}+\widetilde{\Delta}\right\rangle\nonumber\\
    =&\tau^2+\left\langle-\tau\frac{\bm{w}}{\|\bm{w}\|_2},\widetilde{\Delta}\right\rangle
    +\left\langle-\tau\frac{\bm{w}}{\|\bm{w}\|_2},\widehat{\Delta}\right\rangle+\langle\widehat{\Delta},\widetilde{\Delta}\rangle\nonumber\\    
    \geq&\tau^2-\tau\|\widetilde{\Delta}\|_2-\tau\|\widehat{\Delta}\|_2-\|\widehat{\Delta}\|_2\|\widetilde{\Delta}\|_2\nonumber\\   
    \geq&\tau^2-\tau^2\frac{2\alpha\sqrt{d}}{\sqrt{3}(\|\bm{w}\|_2-\alpha\sqrt{d})}-2\tau^2\left(1-\left(1-\frac{d^2\beta^{-2t}}{\|\bm{w}\|^2_2}\right)^{\frac{d-1}{2}}\right)\nonumber\\
    &-\frac{4}{\sqrt{3}}\tau^2\left(1-\left(1-\frac{d^2\beta^{-2t}}{\|\bm{w}\|^2_2}\right)^{\frac{d-1}{2}}\right)\frac{\alpha\sqrt{d}}{(\|\bm{w}\|_2-\alpha\sqrt{d})}\nonumber\\
    \geq&\tau^2\left(1-\frac{2}{\sqrt{3}}\tilde{\rho}-2\hat{\rho}-\frac{4}{\sqrt{3}}\hat{\rho}\tilde{\rho}\right)\nonumber\\
    >&0.49\tau^2\nonumber\\
    =&0.49\tau^2\EE\left[\frac{\|\bm{w}+\bm{u}\|^2_2}{\|\bm{w}+\bm{u}\|^2_2}\right]\nonumber\\ 
    \geq&0.49\|\EE[\widetilde{\nabla} F(\bm{w}+\bm{u},\bm{\xi})]\|^2_2,\nonumber 
	\end{alignat}   	
	\noindent making $c_1=0.49$ valid. Considering now $c_2$ for inequality \eqref{errorupper},
	\begin{alignat}{6}	
	\EE[\|\widehat{\nabla} F(\bm{w},\bm{\xi},\bm{b})\|^2_2]=&\EE[\ind_{\widehat{F}}\langle\bm{z}+\bm{\delta}^{z},\bm{z}+\bm{\delta}^{z}\rangle]\nonumber\\
	\leq&\EE[\|\bm{z}\|^2_2+2\langle \bm{z},\bm{\delta}^{z}\rangle+\|\bm{\delta}^{z}\|^2_2]\nonumber\\  
	\leq&1+2\EE[\|\bm{\delta}^z\|_2]+\EE[\|\bm{\delta}^z\|^2_2]\nonumber\\ 	
	\leq&1+\sqrt{d}\beta^{-t}+\frac{d}{4}\beta^{-2t}\nonumber\\ 	
	<&1.01\nonumber\\
	=&1.01Q\nonumber
	\end{alignat}	
making $c_2=1.01$ valid, where the third inequality used Proposition \ref{storoundprop}, the final inequality used $d=100$, $\beta=2$, and $t=14$, and the final equality used the fact that $Q=\EE[L_0(\bm{\xi})^2]=1$. 
\end{proof}

\subsubsection{Description of a Class of Adaptive Step Sizes $\eta_k$}
\label{adapt_class}

The adaptive step sizes studied in this work are motivated by methods such as gradient normalization and clipping. Besides having the potential to limit the negative effects of numerical error by stabilizing the algorithm steps \eqref{eq:1}, these step sizes require virtually no extra memory, making these light-weight variants of SGD applicable for training with numerical error in environments with limited compute.   

We consider step sizes $\eta_k=\hat{\eta}_k\psi_k$, where $\hat{\eta}_k>0$ is deterministic and $\psi_k\geq0$ is a random variable for all $k\in\NN$. The class of random sequences $\{\psi_k\}$ considered in this work is defined as follows. 

\begin{definition}\label{adaptive_stepsize}
	A sequence $\{\psi_k\}$ is admissible for step sizes $\{\eta_k\}$ if for all $k\in \NN$	
	\begin{enumerate}
		\item $\psi_k$ is essentially bounded by $\F_{k-1}$-measurable random variables $0\leq\Psi^L_k\leq\Psi^U_k<\infty\label{ss-bound}$ conditioning on $\F_{k-1}$: $\PP(\Psi^L_k\leq\psi_k\leq\Psi^U_k|\F_{k-1})=1$, and 
		\item $\Psi^U_k$ is essentially uniformly bounded by constants $0<\underline{\Psi}^U\leq\overline{\Psi}^U<\infty\label{ss-bound-bound}$:  $\PP(\underline{\Psi}^U\leq\Psi^U_k\leq\overline{\Psi}^U)=1$.
	\end{enumerate}
	Let $\Delta_k:=\Psi^U_k-\Psi^L_k$. The sequence $\{\Delta_k\}\label{ss-Del}$ will be used to measure and restrict the level of $\F_k$-measurable adaptiveness of step sizes.  
\end{definition}

\noindent Definition \ref{adaptive_stepsize} allows any essentially bounded random sequence $\{\psi_k\}\subset[0,\overline{\Psi}^U]$. The most common case studied will have $\Delta_k<\overline{\Psi}^U$, where the purpose of $\Psi^L_k$ and $\Psi^U_k$ is to dynamically choose the essential range of $\psi_k$, $[\Psi^L_k,\Psi^U_k]\subset[0,\overline{\Psi}^U]$ through time. A challenge when analyzing $\EE[\psi_k\widehat{\nabla} F^{k}(\bm{w}^k)|\F_{k-1}]$ is that $\psi_k$ can change the expected step direction: 
Consider a 1-dimensional problem without any numerical error, where $\widetilde{\nabla} F^{k}(\bm{w}+\bm{u},\bm{\xi})\in\{5,-1\}$, with $\PP(\widetilde{\nabla} F^{k}(\bm{w}+\bm{u},\bm{\xi})=5)=\frac{1}{4}$ and $\PP(\widetilde{\nabla} F^{k}(\bm{w}+\bm{u},\bm{\xi})=-1)=\frac{3}{4}$. 
It follows that $\nabla_{\alpha} f(\bm{w})=\EE[\widetilde{\nabla} F^{k}(\bm{w}+\bm{u},\bm{\xi})]=\frac{1}{2}$, but $\EE[\psi\widetilde{\nabla} F^{k}(\bm{w}+\bm{u},\bm{\xi})]=-\frac{1}{2}$  when $\psi=\|\widetilde{\nabla} F^{k}(\bm{w}+\bm{u},\bm{\xi})\|^{-1}_2=\min(1,\|\widetilde{\nabla} F^{k}(\bm{w}+\bm{u},\bm{\xi})\|^{-1}_2)\in\{\frac{1}{5},1\}$, i.e., when $\psi$ performs gradient normalization or gradient clipping, PISGD becomes PISGA(scent). By restricting $\psi\in[\frac{4}{5},1]$, this is corrected with $\EE[\psi\widetilde{\nabla} F^{k}(\bm{w}+\bm{u},\bm{\xi})]=\frac{1}{4}$. When adaptive step sizes are not desired, an admissible sequence satisfying Definition \ref{adaptive_stepsize} sets $\psi_k=\Psi^L_k=\Psi^U_k=\underline{\Psi}^U=\overline{\Psi}^U=1$.

There are relevant papers \cite{koloskova2023,zhang2020d,zhang2020c,zhang2020b} which have studied gradient clipping algorithms, proving non-asymptotic convergence results for non-convex stochastic loss functions after running for $J\in\NN$ iterations. The following proposition shows how the gradient clipping algorithm studied in \citep[Theorem 7]{zhang2020c}, 
\begin{alignat}{6}
\bm{w}^{k+1}=\bm{w}^k-\hat{\eta}\min\left(\frac{1}{16\hat{\eta}^2L_1(\|\bm{g}^k\|_2+\sigma)},1\right)\bm{g}^k,\label{grad_clip_alg}
\end{alignat}
fits within Definition \ref{adaptive_stepsize}, where $\bm{g}^k$ is a stochastic gradient of a loss function $f$ sampled at $\bm{w}^k$. In their work, it is assumed that there exists a constant $\sigma>0$ such that $\|\bm{g}-\nabla f(\bm{w})\|_2\leq \sigma$ almost surely for all $\bm{w}\in\RR^d$ \citep[Assumption 5]{zhang2020c}, and that $\hat{\eta}=\min(\frac{1}{20L_0},\frac{1}{128L_1\sigma},\frac{1}{\sqrt{J}})$ \citep[Theorem 7]{zhang2020c}. 

\begin{proposition}
	\label{grad_clip_prop}	
	For the gradient clipping algorithm \eqref{grad_clip_alg} studied in \citep[Theorem 7]{zhang2020c}, the step sizes satisfy Definition \ref{adaptive_stepsize}. If there exists a constant $G>0$ such that either 
	\begin{enumerate}	
	\item  $\|\bm{g}^k\|_2\leq G$ almost surely for all $k\in\NN$, or 
	\item there exists a $J'\in\NN$ such that for $k\geq J'$ $\|\nabla f(\bm{w}^k)\|_2\leq G$ almost surely, 
	\end{enumerate}
	then taking $J\in \NN$ sufficiently large, $\Psi^L_k=\Psi^U_k=1$ with $\Delta_k=0$ becomes valid for (1) all $k\in\NN$ or (2) all $k\geq J'$.
\end{proposition}
\begin{proof}	
	Taking $\psi_k=\min\left(\frac{1}{16\hat{\eta}^2L_1(\|\bm{g}^k\|_2+\sigma)},1\right)$, $\Psi^L_k=0$ and $\Psi^U_k=\underline{\Psi}^U=\overline{\Psi}^U=1$ for $k\in\NN$ are valid. When the gradient is not clipped, i.e., $\frac{1}{16\hat{\eta}^2L_1(\|\bm{g}^k\|_2+\sigma)}\geq1$, \eqref{grad_clip_alg} takes the form of SGD with step size $\hat{\eta}$. If there exists a $J'\in\NN$ such that gradient clipping does not occur almost surely for $k\geq J'$, then for $k\geq J'$ $\Psi^L_k=1$ is valid and $\Delta_k=0$. What remains to show is that this occurs when either conditions (1) or (2) hold and $J$ is taken sufficiently large.
	
	Given that $\hat{\eta}=\min\left(\frac{1}{20L_0},\frac{1}{128L_1\sigma},\frac{1}{\sqrt{J}}\right)$, for $J$ sufficiently large $\hat{\eta}=\frac{1}{\sqrt{J}}$ and
	$\psi_k=\min\left(\frac{J}{16L_1(\|\bm{g}^k\|_2+\sigma)},1\right)$. If condition (1) holds, taking $J$ sufficiently large such that $\frac{J}{16L_1(G+\sigma)}\geq1$, no gradient clipping will be performed almost surely for all $k\in\NN$.
	
	If condition (2) holds, using \citep[Assumption 5]{zhang2020c} described above this proposition: For all $\bm{w}\in\RR^d$, almost surely,
	\begin{alignat}{6}
		&&\sigma&\geq \|\bm{g}-\nabla f(\bm{w})\|_2\nonumber\\ 
		&&&\geq\|\bm{g}\|_2-\|\nabla f(\bm{w})\|_2\nonumber\\
		\Rightarrow&&\|\bm{g}\|_2&\leq \sigma+\|\nabla f(\bm{w})\|_2\text{, hence}\nonumber
	\end{alignat}	
	\begin{alignat}{6}
		&&\frac{J}{16L_1(\|\bm{g}^k\|_2+\sigma)}\geq \frac{J}{16L_1(\|\nabla f(\bm{w}^k)\|_2+2\sigma)}\nonumber
	\end{alignat}
	almost surely. Setting $J\geq 16L_1(G+2\sigma)$, it holds almost surely for $k\geq J'$ that 
	\begin{alignat}{6}
		&&\psi_k\geq& \min\left(\frac{J}{16L_1(\|\nabla f(\bm{w}^k)\|_2+2\sigma)},1\right)\nonumber\\
		&&\geq& \min\left(\frac{16L_1(G+2\sigma)}{16L_1(G+2\sigma)},1\right)\geq 1,\nonumber
	\end{alignat} 
	with no gradient clipping being performed.
\end{proof}

\noindent The same style of proposition can be proven showing that the other gradient clipping algorithms stop clipping after a finite number of iterations, and hence lose all adaptiveness in their step sizes if condition (1) \cite{koloskova2023,zhang2020b} or (2) \cite{zhang2020d} holds. This surface-level analysis of previous gradient clipping algorithms further points to the difficulty of proving the unrestricted convergence of adaptive step size methods for non-convex problems. 

\subsubsection{Assumptions Concerning $\hat{\bm{e}}^k$}

The random vector $\hat{\bm{e}}^k\in\RR^d$ in \eqref{eq:1} models the error from computing the addition, subtraction, multiplication, and division with finite precision in \eqref{eq:1} given $S^k:=\{\bm{w}^k,\hat{\eta}_k,\psi_k,M,\{\widehat{\nabla} F^{k,i}(\bm{w}^k)\}\}$. 
The convergence analysis requires that the expected value of $\hat{\bm{e}}^k$ equals 0 when conditioned on $\sigma(\F_{k-1},\G_k)$ and that $\EE[\|\hat{\bm{e}}^k\|^2_2|\F_{k-1}]=O(\hat{\eta}^2_k)$.
\begin{assumption}\label{e_assump}
	There exists a constant $c_3>0$ and a $K\in\NN$ such that for all $k\geq K$, almost surely 	
	\begin{alignat}{6}
		&&\EE[\hat{\bm{e}}^k|\F_{k-1},\G_k]&=\bm{0}\quad\text{and}\quad\EE[\|\hat{\bm{e}}^k\|^2_2|\F_{k-1}]&\leq c_3\hat{\eta}^2_k.\nonumber
\end{alignat}
\end{assumption}

\noindent It will now be shown how Assumption \ref{e_assump} holds in a fixed-point environment $\FF$ using stochastic rounding.

\begin{proposition}
	\label{fixed-point}	
	Let
	\begin{alignat}{6}
	&\bm{w}^k\ominus((\hat{\eta}_k\otimes\psi_k)\oslash M)\otimes(\widehat{\nabla} F^{k,1}(\bm{w}^k)\oplus...\oplus\widehat{\nabla} F^{k,M}(\bm{w}^k))\nonumber\\
	=&\bm{w}^k-\frac{\hat{\eta}_k\psi_k}{M}\sum_{i=1}^M\widehat{\nabla} F^{k,i}(\bm{w}^k)+\hat{\bm{e}}^k,\label{fixedsgd}
\end{alignat}
	where the `o' symbols represent the corresponding operation in a fixed-point environment $\FF$ using stochastic rounding. Assume that $\bm{w}^k,\widehat{\nabla} F^{k,i}(\bm{w}^k)\in \FF^d$ for $i\in [M]$, $\hat{\eta}_k,M\in\FF_{>0}$, $\psi_k\in\FF_{\geq0}$, $r\geq 0$ in \eqref{fixp} is chosen such that no overflow will occur, and that $k\in\NN$ is sufficiently large such that Proposition \ref{normbound} holds. Assumption \ref{e_assump} holds with $c_3=\frac{1}{4}((M^2+1)c_2Q+M)$. 
\end{proposition}

\begin{proof}Evaluating the left-hand side of \eqref{fixedsgd}, following the order of operations, and using the rounding error bounds given in Section \ref{finiteprec}, 
	\begin{alignat}{6}
		&&&\bm{w}^k\ominus((\hat{\eta}_k\otimes\psi_k)\oslash M){\otimes}(\widehat{\nabla} F^{k,1}(\bm{w}^k)\oplus...\oplus\widehat{\nabla} F^{k,M}(\bm{w}^k))\nonumber\\	
		&&=&\bm{w}^k\ominus((\hat{\eta}_k\psi_k+\delta_0)\oslash M){\otimes}(\widehat{\nabla} F^{k,1}(\bm{w}^k)\oplus...\oplus\widehat{\nabla} F^{k,M}(\bm{w}^k))\nonumber\\			
		&&=&\bm{w}^k\ominus\left(\frac{\hat{\eta}_k\psi_k+\delta_0}{M}+\delta_1\right){\otimes}(\widehat{\nabla} F^{k,1}(\bm{w}^k)\oplus...\oplus\widehat{\nabla} F^{k,M}(\bm{w}^k))\nonumber\\
		&&=&\bm{w}^k\ominus\left(\frac{\hat{\eta}_k\psi_k+\delta_0}{M}+\delta_1\right){\otimes}\sum_{i=1}^M\widehat{\nabla} F^{k,i}(\bm{w}^k),\label{break1}
	\end{alignat}
	where $\delta_0\in\RR$ is the rounding error from the first multiplication and $\delta_1\in\RR$ is the error from the division. Equation \eqref{break1} is equal to 
	\begin{alignat}{6}	
	&&&\bm{w}^k\ominus\left(\left(\frac{\hat{\eta}_k\psi_k+\delta_0}{M}+\delta_1\right)\sum_{i=1}^M\widehat{\nabla} F^{k,i}(\bm{w}^k)+\bm{\delta}^{(2)}\right)\nonumber\\
	&&=&\bm{w}^k-\left(\left(\frac{\hat{\eta}_k\psi_k+\delta_0}{M}+\delta_1\right)\sum_{i=1}^M\widehat{\nabla} F^{k,i}(\bm{w}^k)+\bm{\delta}^{(2)}\right)\nonumber\\
	&&=&\bm{w}^k-\frac{\hat{\eta}_k\psi_k}{M}\sum_{i=1}^M\widehat{\nabla} F^{k,i}(\bm{w}^k)-\left(\frac{\delta_0}{M}+\delta_1\right)\sum_{i=1}^M\widehat{\nabla} F^{k,i}(\bm{w}^k)-\bm{\delta}^{(2)},\nonumber	
	\end{alignat}	
	where $\bm{\delta}^{(2)}\in\RR^d$ is the vector of errors from the multiplication. Setting $\hat{\bm{e}}^k=-\left(\frac{\delta_0}{M}+\delta_1\right)\sum_{i=1}^M\widehat{\nabla} F^{k,i}(\bm{w}^k)-\bm{\delta}^{(2)}$,  
	\begin{alignat}{6}
		&&&\EE[\hat{\bm{e}}^k|\F_{k-1},\G_k]\nonumber\\
		&&=&-\EE[(\delta_0+M\delta_1)\widehat{\nabla} \overline{F}^k(\bm{w}^k)|\F_{k-1},\G_k]
		-\EE[\EE[\bm{\delta}^{(2)}|\F_{k-1},\G_k,\delta_0,\delta_1]|\F_{k-1},\G_k]\nonumber\\	
		&&=&-\EE[(\delta_0+M\delta_1)|\F_{k-1},\G_k]\widehat{\nabla} \overline{F}^k(\bm{w}^k)\nonumber\\	
		&&=&-M\EE[\EE[\delta_1|\F_{k-1},\G_k,\delta_0]|\F_{k-1},\G_k]\widehat{\nabla} \overline{F}^k(\bm{w}^k)=0.\nonumber	
	\end{alignat}		
	Considering now $\EE[\|\hat{\bm{e}}^k\|^2_2|\F_{k-1}]$,	
	\begin{alignat}{6}
		&&&\EE[\|\hat{\bm{e}}^k\|^2_2|\F_{k-1}]\nonumber\\
		&&=&\EE[\|(\delta_0+M\delta_1)\widehat{\nabla} \overline{F}^k(\bm{w}^k)+\bm{\delta}^{(2)}\|^2_2|\F_{k-1}]\nonumber\\	
		&&=&\EE[\|(\delta_0+M\delta_1)\widehat{\nabla} \overline{F}^k(\bm{w}^k)\|^2_2|\F_{k-1}]+2\EE[\langle (\delta_0+M\delta_1)\widehat{\nabla} \overline{F}^k(\bm{w}^k), \bm{\delta}^{(2)}\rangle|\F_{k-1}]\nonumber\\ &&+&\EE[\|\bm{\delta}^{(2)}\|^2_2|\F_{k-1}].\label{fixedineq}	
	\end{alignat}	
	Focusing on the first term,  
	\begin{alignat}{6}
		&&&\EE[(\delta_0+M\delta_1)^2\|\widehat{\nabla} \overline{F}^k(\bm{w}^k)\|^2_2|\F_{k-1}]\nonumber\\
		&&=&\EE[\EE[(\delta_0+M\delta_1)^2|\F_{k-1},\G_k]\|\widehat{\nabla} \overline{F}^k(\bm{w}^k)\|^2_2|\F_{k-1}]\nonumber\\
		&&\leq&(M^2+1)\frac{\beta^{-2t}}{4}\EE[\|\widehat{\nabla} \overline{F}^k(\bm{w}^k)\|^2_2|\F_{k-1}]\nonumber\\
		&&\stackrel{\text{a.s.}}{\leq}&(M^2+1)\frac{\beta^{-2t}}{4}c_2Q,\nonumber
	\end{alignat}	
	using Propositions \ref{storoundprop} and Proposition \ref{normbound}, where 
		\begin{alignat}{6}
		&&&\EE[\delta^2_0+2\delta_0M\delta_1+M^2\delta^2_1|\F_{k-1},\G_k]\nonumber\\
		&&\leq&\frac{\beta^{-2t}}{4}+\EE[\EE[2\delta_0M\delta_1+M^2\delta^2_1|\F_{k-1},\G_k,\delta_0]|\F_{k-1},\G_k]\nonumber\\				
		&&\leq&(M^2+1)\frac{\beta^{-2t}}{4}.\nonumber
	\end{alignat}	
	Considering now the second term of \eqref{fixedineq}, 	
	\begin{alignat}{6}	
		&&&2\EE[\langle (\delta_0+M\delta_1)\widehat{\nabla} \overline{F}^k(\bm{w}^k), \bm{\delta}^{(2)}\rangle|\F_{k-1}]\nonumber\\
		&&=&2\EE[\EE[\langle (\delta_0+M\delta_1)\widehat{\nabla} \overline{F}^k(\bm{w}^k), \bm{\delta}^{(2)}\rangle|\F_{k-1},\G_k,\delta_0,\delta_1]|\F_{k-1}]\nonumber\\
		&&=&2\EE[\langle (\delta_0+M\delta_1)\widehat{\nabla} \overline{F}^k(\bm{w}^k),\EE[\bm{\delta}^{(2)}|\F_{k-1},\G_k,\delta_0,\delta_1] \rangle|\F_{k-1}]=0,\nonumber
	\end{alignat}
	and the final term, 
	\begin{alignat}{6}
		\EE[\|\bm{\delta}^{(2)}\|^2_2|\F_{k-1}]=\sum_{i=1}^M\EE[\EE[(\delta^{(2)}_i)^2|\F_{k-1},\G_k,\delta_0,\delta_1]|\F_{k-1}]
		\leq M\frac{\beta^{-2t}}{4}.\nonumber
	\end{alignat}		
	Continuing from \eqref{fixedineq}, 
	\begin{alignat}{6}
		\EE[\|\hat{\bm{e}}^k\|^2_2|\F_{k-1}]
		&&\stackrel{\text{a.s.}}{\leq}&((M^2+1)c_2Q+M)\frac{\beta^{-2t}}{4}\leq ((M^2+1)c_2Q+M)\frac{\hat{\eta}^2_k}{4},\nonumber
	\end{alignat}
	where the second inequality holds since $\lambda=\beta^{-t}\leq\hat{\eta}_k\in\FF_{>0}$.  	
\end{proof}

\subsection{Convergence Analysis of PISGD with Numerical Error}

This section presents the asymptotic convergence result to a Clarke stationary point. The convergence analysis considers the following subclass of sequences $\{\psi_k\}$ satisfying Definition \ref{adaptive_stepsize}.

\begin{definition}\label{delta_bound}
	A sequence $\{\psi_k\}$ is admissible for step sizes $\{\eta_k\}$ in Theorem \ref{convergence} if it satisfies Definition \ref{adaptive_stepsize}, and there exists constants $c_4>0$, $\Delta_{c_1}>0\label{del_c1}$, and a $K\in\NN$ such that for all $k\geq K$, almost surely 
		$(1)\text{ }\Delta_k\leq c_4\frac{\hat{\eta}_k}{\alpha_k}$ and $(2)\text{ }\Delta_k\leq 2c_1\Psi^U_k-\Delta_{c_1}$.		 
\end{definition}

 \noindent When $c_1>0.5$, $2c_1\Psi^U_k-\Delta_k\geq (2c_1-1)\Psi^U_k\geq(2c_1-1)\underline{\Psi}^U>0$, hence $\Delta_{c_1}\in(0,(2c_1-1)\underline{\Psi}^U]$ always exists. For the problem studied in Proposition \ref{ex_ramp} with $c_1=0.49$, by choosing $\Delta_{c_1}=0.01$, a restricted gradient clipping of the form $\psi_k=\max(0.03,\min(1,\|\widehat{\nabla} \overline{F}^k(\bm{w}^k)\|^{-1}_2))$ satisfies Definition \ref{delta_bound}(2) by taking $\Psi^L_k=0.03$ and $\Psi^U_k=1$ for all $k\in\NN$. 

\begin{theorem}
	\label{convergence}	
	 Let $\{\alpha_k\}$ and $\{\hat{\eta}_k\}$ be chosen such that
	\begin{alignat}{6}
		\sum_{k=1}^{\infty}\alpha_k^d\hat{\eta}_k=\infty\quad\text{and}\quad \sum_{k=1}^{\infty}\alpha_k^{d-1}\hat{\eta}^2_k<\infty,\label{the_cond}
	\end{alignat}	
	where $\{\alpha_k\}$ is non-increasing and $\lim\limits_{k\rightarrow \infty}\alpha_k=0$. Let $\{\psi_k\}$ satisfy Definition \ref{delta_bound}. If Assumptions \ref{approx_assump} and \ref{e_assump} hold, PISGD \eqref{eq:1} almost surely generates a subsequence of indices $\{k_i\}$ such that 
	$$\lim\limits_{i\rightarrow \infty} \|\nabla f_{\alpha_{k_i}}(\bm{w}^{k_i})\|_2=0,$$
	and for every accumulation point $\bm{w}^*$ of $\{\bm{w}^{k_i}\}$,  
	\begin{alignat}{6}
		\dist(\bm{0},\partial f(\bm{w}^*))=0.\nonumber
	\end{alignat}	
\end{theorem}

\noindent The proof of Theorem \ref{convergence} requires the following Robbins-Siegmund inequality. 
\begin{lemma}{\cite[Theorem 1]{robbins1971}}\label{robbins}
	For all $k\in\NN$, let $z_k$, $\theta_k$, and $\zeta_k$ be non-negative $\F_{k-1}$-measurable random variables such that almost surely 
	\begin{alignat}{6}
		&&\EE[z_{k+1}|\F_{k-1}]&\leq z_k+\theta_k-\zeta_k\nonumber
	\end{alignat} 
	and $\sum_{k=1}^\infty\theta_k<\infty$. It holds almost surely that $\sum_{k=1}^\infty\zeta_k<\infty$.
\end{lemma}

\begin{proof}{\it (Theorem \ref{convergence}):} Let the analysis begin at $k=\overline{K}\in \NN$, where $\overline{K}$ is sufficiently large such that for all $k'\geq \overline{K}$ the (in)equalities in Assumptions \ref{approx_assump} and \ref{e_assump}, and Definition \ref{delta_bound} hold. By the $L^{\alpha}_1$-smoothness of $f_{\alpha}$ (Proposition \ref{old_props}.2 \& \cite[Lemma 1.2.3]{nesterov2004}),
	\begin{alignat}{6}
		&&f_{\alpha_k}(\bm{w}^{k+1})&\leq f_{\alpha_k}(\bm{w}^k)+\langle \nabla f_{\alpha_k}(\bm{w}^k),\bm{w}^{k+1}-\bm{w}^k\rangle+\frac{L_1^{\alpha_k}}{2}\|\bm{w}^{k+1}-\bm{w}^k\|^2_2\nonumber\\
		&&&=f_{\alpha_k}(\bm{w}^k)+\langle \nabla f_{\alpha_k}(\bm{w}^k),-\hat{\eta}_k\psi_k\widehat{\nabla} \overline{F}^k(\bm{w}^k)+\hat{\bm{e}}^k\rangle+\frac{L_1^{\alpha_k}}{2}\|\bm{w}^{k+1}-\bm{w}^k\|^2_2\label{non-start}\\
		\Rightarrow&&f_{\alpha_{k+1}}(\bm{w}^{k+1})&\leq f_{\alpha_k}(\bm{w}^k)+f_{\alpha_{k+1}}(\bm{w}^{k+1})-f_{\alpha_k}(\bm{w}^{k+1})-\hat{\eta}_k\psi_k\langle \nabla f_{\alpha_k}(\bm{w}^k),\widehat{\nabla} \overline{F}^k(\bm{w}^k)\rangle\nonumber\\
		&&&+\langle \nabla f_{\alpha_k}(\bm{w}^k),\hat{\bm{e}}^k\rangle+\frac{L_1^{\alpha_k}}{2}\|\bm{w}^{k+1}-\bm{w}^k\|^2_2.\label{ineq:1}
	\end{alignat} 
	Focusing on $f_{\alpha_{k+1}}(\bm{w}^{k+1})-f_{\alpha_k}(\bm{w}^{k+1})$,    
	\begin{alignat}{6}
		&&&f_{\alpha_{k+1}}(\bm{w}^{k+1})-f_{\alpha_k}(\bm{w}^{k+1})\nonumber\\
		&&=&f_{\alpha_{k+1}}(\bm{w}^{k+1})-\int_{-\alpha_k}^{\alpha_k}\int_{-\alpha_{k}}^{\alpha_{k}}...\int_{-\alpha_{k}}^{\alpha_{k}}\frac{f(\bm{w}^{k+1}+\bm{u})}{(2\alpha_k)^d}du_1du_2...du_d\nonumber\\
		&&=&f_{\alpha_{k+1}}(\bm{w}^{k+1})-\int_{-\alpha_k}^{\alpha_k}\int_{-\alpha_{k}}^{\alpha_{k}}...\int_{-\alpha_{k}}^{\alpha_{k}}\ind_{\{\bm{u}\in\RR^d:\|\bm{u}\|_{\infty}\leq\alpha_{k+1}\}}\frac{f(\bm{w}^{k+1}+\bm{u})}{(2\alpha_k)^d}du_1du_2...du_d\nonumber\\
		&&&-\int_{-\alpha_k}^{\alpha_k}\int_{-\alpha_{k}}^{\alpha_{k}}...\int_{-\alpha_{k}}^{\alpha_{k}}\ind_{\{\bm{u}\in\RR^d:\|\bm{u}\|_{\infty}>\alpha_{k+1}\}}\frac{f(\bm{w}^{k+1}+\bm{u})}{(2\alpha_k)^d}du_1du_2...du_d\nonumber\\
		&&=&f_{\alpha_{k+1}}(\bm{w}^{k+1})-f_{\alpha_{k+1}}(\bm{w}^{k+1})\frac{\alpha_{k+1}^d}{\alpha_k^d}\nonumber\\
		&&&-\int_{-\alpha_k}^{\alpha_k}\int_{-\alpha_{k}}^{\alpha_{k}}...\int_{-\alpha_{k}}^{\alpha_{k}}
		\ind_{\{\bm{u}\in\RR^d:\|\bm{u}\|_{\infty}>\alpha_{k+1}\}}\frac{f(\bm{w}^{k+1}+\bm{u})}{(2\alpha_k)^d}du_1du_2...du_d\nonumber\\
		&&\leq&f_{\alpha_{k+1}}(\bm{w}^{k+1})\left(1-\frac{\alpha_{k+1}^d}{\alpha_k^d}\right),\nonumber
	\end{alignat} 
	where the assumption that $\alpha_{k+1}\leq\alpha_k$ was used for the third equality, and Assumption \ref{pos_loss} was used for the final inequality. Plugging into \eqref{ineq:1},
	\begin{alignat}{6}
		&f_{\alpha_{k+1}}(\bm{w}^{k+1})\nonumber\\
		\leq &f_{\alpha_k}(\bm{w}^k)+f_{\alpha_{k+1}}(\bm{w}^{k+1})\left(1-\frac{\alpha_{k+1}^d}{\alpha_k^d}\right)
		-\hat{\eta}_k\psi_k\langle \nabla f_{\alpha_k}(\bm{w}^k),\widehat{\nabla} \overline{F}^k(\bm{w}^k)\rangle\nonumber\\
		+&\langle \nabla f_{\alpha_k}(\bm{w}^k),\hat{\bm{e}}^k\rangle+\frac{L_1^{\alpha_k}}{2}\|\bm{w}^{k+1}-\bm{w}^k\|^2_2\nonumber\\
		\Rightarrow\frac{\alpha_{k+1}^d}{\alpha_k^d}&f_{\alpha_{k+1}}(\bm{w}^{k+1})\nonumber\\
		\leq& f_{\alpha_k}(\bm{w}^k)-\hat{\eta}_k\psi_k\langle \nabla f_{\alpha_k}(\bm{w}^k),\widehat{\nabla} \overline{F}^k(\bm{w}^k)\rangle\nonumber\\
		+&\langle \nabla f_{\alpha_k}(\bm{w}^k),\hat{\bm{e}}^k\rangle+\frac{\sqrt{d}L_0}{2\alpha_k}\|-\hat{\eta}_k\psi_k\widehat{\nabla} \overline{F}^k(\bm{w}^k)+\hat{\bm{e}}^k\|^2_2\nonumber\\
		\Rightarrow\alpha_{k+1}^d&f_{\alpha_{k+1}}(\bm{w}^{k+1})\nonumber\\
		\leq&\alpha_k^df_{\alpha_k}(\bm{w}^k)-\alpha_k^d\hat{\eta}_k\psi_k\langle \nabla f_{\alpha_k}(\bm{w}^k),\widehat{\nabla} \overline{F}^k(\bm{w}^k)\rangle+\alpha_k^d\langle \nabla f_{\alpha_k}(\bm{w}^k),\hat{\bm{e}}^k\rangle\nonumber\\
		+&\frac{\alpha_k^{d-1}\sqrt{d}L_0}{2}(\hat{\eta}^2_k\psi^2_k\|\widehat{\nabla} \overline{F}^k(\bm{w}^k)\|^2_2-2\hat{\eta}_k\psi_k\langle \widehat{\nabla} \overline{F}^k(\bm{w}^k),\hat{\bm{e}}^k\rangle+\|\hat{\bm{e}}^k\|^2_2),\label{ineq:2}
	\end{alignat} 
	where the value of $L_1^{\alpha_k}$ from Proposition \ref{old_props} was used in the second inequality. Taking the conditional expectation of \eqref{ineq:2} with respect to $\F_{k-1}$,
	\begin{alignat}{6}
		&\EE[\alpha_{k+1}^df_{\alpha_{k+1}}(\bm{w}^{k+1})|\F_{k-1}]\nonumber\\
		\leq&\alpha_k^df_{\alpha_k}(\bm{w}^k)-\alpha_k^d\hat{\eta}_k\EE[\psi_k \langle \nabla f_{\alpha_k}(\bm{w}^k),\widehat{\nabla} \overline{F}^k(\bm{w}^k)\rangle|\F_{k-1}]+\alpha_k^d\langle \nabla f_{\alpha_k}(\bm{w}^k),\EE[\hat{\bm{e}}^k|\F_{k-1}]\rangle\nonumber\\
		+&\frac{\alpha_k^{d-1}\sqrt{d}L_0}{2}(\hat{\eta}^2_k\EE[\psi^2_k\|\widehat{\nabla} \overline{F}^k(\bm{w}^k)\|^2_2|\F_{k-1}]-2\hat{\eta}_k\EE[\psi_k\langle \widehat{\nabla} \overline{F}^k(\bm{w}^k),\hat{\bm{e}}^k\rangle|\F_{k-1}]\nonumber\\
		&\quad\quad\quad\quad\quad+\EE[\|\hat{\bm{e}}^k\|^2_2|\F_{k-1}]).\label{prerobbsieg}
	\end{alignat}
	It holds that $\EE[\hat{\bm{e}}^k|\F_{k-1}]=\EE[\EE[\hat{\bm{e}}^k|\F_{k-1},\G_k]|\F_{k-1}]=\bm{0}$ almost surely by Assumption \ref{e_assump}. By Definition \ref{adaptive_stepsize} and Proposition \ref{normbound}, 
	\begin{alignat}{6}
		&\EE[\psi^2_k\|\widehat{\nabla} \overline{F}^k(\bm{w}^k)\|^2_2|\F_{k-1}]\stackrel{\text{a.s.}}{\leq} (\Psi^U_k)^2\EE[\|\widehat{\nabla} \overline{F}^k(\bm{w}^k)\|^2_2|\F_{k-1}]
		\stackrel{\text{a.s.}}{\leq}(\overline{\Psi}^U)^2c_2Q.\nonumber
	\end{alignat}	
	Using Assumption \ref{e_assump},	
	\begin{alignat}{6}
		&&&\EE[\psi_k\langle \widehat{\nabla} \overline{F}^k(\bm{w}^k),\hat{\bm{e}}^k\rangle|\F_{k-1}]\nonumber\\
		&&=&\EE[\EE[\psi_k\langle \widehat{\nabla}\overline{F}^k(\bm{w}^k),\hat{\bm{e}}^k\rangle|\G_k,\F_{k-1}]|\F_{k-1}]\nonumber\\
		&&=&\EE[\psi_k\langle \widehat{\nabla} \overline{F}^k(\bm{w}^k),\EE[\hat{\bm{e}}^k|\G_k,\F_{k-1}]\rangle|\F_{k-1}]\stackrel{\text{a.s.}}{=}0\nonumber	
	\end{alignat}
	and $\EE[\|\hat{\bm{e}}^k\|^2_2|\F_{k-1}]\leq c_3 \hat{\eta}^2_k$ almost surely. Applying these (in)equalities in \eqref{prerobbsieg}, 	
	\begin{alignat}{6}
		&&\EE[\alpha_{k+1}^df_{\alpha_{k+1}}(\bm{w}^{k+1})|\F_{k-1}]&\stackrel{\text{a.s.}}{\leq} \alpha_k^df_{\alpha_k}(\bm{w}^k)-\alpha_k^d\hat{\eta}_k\EE[\psi_k\langle \nabla f_{\alpha_k}(\bm{w}^k),\widehat{\nabla} \overline{F}^k(\bm{w}^k)\rangle|\F_{k-1}]\nonumber\\
		&&&+\frac{\alpha_k^{d-1}\hat{\eta}^2_k\sqrt{d}L_0}{2}((\overline{\Psi}^U)^2c_2Q+c_3). \label{prerobbsieg2}
	\end{alignat}
	Focusing on the conditional expectation $\EE[-\psi_k\langle \nabla f_{\alpha_k}(\bm{w}^k),\widehat{\nabla} \overline{F}^k(\bm{w}^k)\rangle|\F_{k-1}]$:
	\begin{alignat}{6}
		&&&\EE[-\psi_k\langle \nabla f_{\alpha_k}(\bm{w}^k),\widehat{\nabla} \overline{F}^k(\bm{w}^k)\rangle|\F_{k-1}]\nonumber\\
		&&=&\EE\left[\frac{\psi_k}{2}(\|\nabla f_{\alpha_k}(\bm{w}^k)-\widehat{\nabla} \overline{F}^k(\bm{w}^k)\|^2_2-\|\nabla f_{\alpha_k}(\bm{w}^k)\|^2_2-\|\widehat{\nabla} \overline{F}^k(\bm{w}^k)\|^2_2)\middle|\F_{k-1}\right]\nonumber\\
		&&\stackrel{\text{a.s.}}{\leq}&\frac{\Psi^U_k}{2}\EE[\|\nabla f_{\alpha_k}(\bm{w}^k)-\widehat{\nabla} \overline{F}^k(\bm{w}^k)\|^2_2|\F_{k-1}]-\frac{\Psi^L_k}{2}\|\nabla f_{\alpha_k}(\bm{w}^k)\|^2_2\nonumber\\
		&&-&\frac{\Psi^L_k}{2}\EE[\|\widehat{\nabla} \overline{F}^k(\bm{w}^k)\|^2_2|\F_{k-1}]\label{var_reduce}\\	
		&&=&\frac{\Psi^U_k}{2}(\|\nabla f_{\alpha_k}(\bm{w}^k)\|^2_2-2\langle \nabla f_{\alpha_k}(\bm{w}^k),\EE[\widehat{\nabla} \overline{F}^k(\bm{w}^k)|\F_{k-1}]\rangle\nonumber\\
		&&+&\EE[\|\widehat{\nabla} \overline{F}^k(\bm{w}^k)\|^2_2|\F_{k-1}])
		-\frac{\Psi^L_k}{2}\|\nabla f_{\alpha_k}(\bm{w}^k)\|^2_2-\frac{\Psi^L_k}{2}\EE[\|\widehat{\nabla} \overline{F}^k(\bm{w}^k)\|^2_2|\F_{k-1}]\nonumber\\
		&&\stackrel{\text{a.s.}}{\leq}&\frac{\Psi^U_k}{2}(\|\nabla f_{\alpha_k}(\bm{w}^k)\|^2_2-2c_1\|\nabla f_{\alpha_k}(\bm{w}^k)\|^2_2+\EE[\|\widehat{\nabla} \overline{F}^k(\bm{w}^k)\|^2_2|\F_{k-1}])\nonumber\\
		&&-&\frac{\Psi^L_k}{2}\|\nabla f_{\alpha_k}(\bm{w}^k)\|^2_2-\frac{\Psi^L_k}{2}\EE[\|\widehat{\nabla} \overline{F}^k(\bm{w}^k)\|^2_2|\F_{k-1}]\nonumber\\
		&&=&\left(\frac{\Psi^U_k}{2}-c_1\Psi^U_k-\frac{\Psi^L_k}{2}\right)\|\nabla f_{\alpha_k}(\bm{w}^k)\|^2_2	
		+\left(\frac{\Psi^U_k}{2}-\frac{\Psi^L_k}{2}\right)\EE[\|\widehat{\nabla} \overline{F}^k(\bm{w}^k)\|^2_2|\F_{k-1}]\nonumber\\
		&&\stackrel{\text{a.s.}}{\leq}&\left(\frac{\Delta_k}{2}-c_1\Psi^U_k\right)\|\nabla f_{\alpha_k}(\bm{w}^k)\|^2_2	
		+\frac{\Delta_k}{2}c_2Q\label{pre_cross_term}\\	
		&&\stackrel{\text{a.s.}}{\leq}&-\frac{\Delta_{c_1}}{2}\|\nabla f_{\alpha_k}(\bm{w}^k)\|^2_2	
		+\frac{c_4}{2}\frac{\hat{\eta}_k}{\alpha_k}c_2Q,\label{cross_term}
	\end{alignat}
	where the first inequality uses Definition \ref{adaptive_stepsize}(1), the second inequality uses inequality \eqref{errorineq} of Assumption \ref{approx_assump}, the third inequality uses Proposition \ref{normbound}, and the last inequality uses Definition \ref{delta_bound}. Plugging \eqref{cross_term} into \eqref{prerobbsieg2}, 
	\begin{alignat}{6}
		&&&\EE[\alpha_{k+1}^df_{\alpha_{k+1}}(\bm{w}^{k+1})|\F_{k-1}]\nonumber\\
		&&\stackrel{\text{a.s.}}{\leq}& \alpha_k^df_{\alpha_k}(\bm{w}^k)-\alpha_k^d\hat{\eta}_k\left(\frac{\Delta_{c_1}}{2}\|\nabla f_{\alpha_k}(\bm{w}^k)\|^2_2	
		-\frac{c_4}{2}\frac{\hat{\eta}_k}{\alpha_k}c_2Q\right)
		+\frac{\alpha_k^{d-1}\hat{\eta}^2_k\sqrt{d}L_0}{2}((\overline{\Psi}^U)^2c_2Q+c_3)\nonumber\\
		&&=&\alpha_k^df_{\alpha_k}(\bm{w}^k)-\alpha_k^d\hat{\eta}_k\frac{\Delta_{c_1}}{2}\|\nabla f_{\alpha_k}(\bm{w}^k)\|^2_2+\frac{\alpha_k^{d-1}\hat{\eta}^2_k}{2}c_2c_4Q\nonumber\\
		&&+&\frac{\alpha_k^{d-1}\hat{\eta}^2_k\sqrt{d}L_0}{2}((\overline{\Psi}^U)^2c_2Q+c_3)\nonumber\\
		&&=&\alpha_k^df_{\alpha_k}(\bm{w}^k)-\alpha_k^d\hat{\eta}_k\frac{\Delta_{c_1}}{2}\|\nabla f_{\alpha_k}(\bm{w}^k)\|^2_2
		+\frac{\alpha_k^{d-1}\hat{\eta}^2_k}{2}(\sqrt{d}L_0((\overline{\Psi}^U)^2c_2Q+c_3)+c_2c_4Q).\nonumber
	\end{alignat}
	
	\noindent Lemma \ref{robbins} can now be applied with $z_k=\alpha_k^df_{\alpha_k}(\bm{w}^k)$, $\theta_k=\frac{\alpha_k^{d-1}\hat{\eta}^2_k}{2}(\sqrt{d}L_0((\overline{\Psi}^U)^2c_2Q+c_3)+c_2c_4Q)$, and $\zeta_k=\alpha_k^d\hat{\eta}_k\frac{\Delta_{c_1}}{2}\|\nabla f_{\alpha_k}(\bm{w}^k)\|^2_2$, given that 
	\begin{alignat}{6}
		&&&\sum_{k=\overline{K}}^{\infty}\frac{\alpha_k^{d-1}\hat{\eta}^2_k}{2}(\sqrt{d}L_0((\overline{\Psi}^U)^2c_2Q+c_3)+c_2c_4Q)\nonumber\\
		&&\leq&\frac{1}{2}(\sqrt{d}L_0((\overline{\Psi}^U)^2c_2Q+c_3)+c_2c_4Q)\sum_{k=1}^{\infty}\alpha_k^{d-1}\hat{\eta}^2_k<\infty\nonumber
	\end{alignat}	
	by assumption, proving that almost surely	
	\begin{alignat}{6}
		\sum_{k=\overline{K}}^{\infty}\alpha_k^d\hat{\eta}_k\frac{\Delta_{c_1}}{2}\|\nabla f_{\alpha_k}(\bm{w}^k)\|^2_2<\infty.\label{bound}
	\end{alignat}
	
	\noindent It follows that $\liminf\limits_{k\rightarrow \infty} \|\nabla f_{\alpha_k}(\bm{w}^k)\|_2=0$ almost surely, given that for any $\epsilon>0$ if there exists a $\overline{K}_2\geq \overline{K}$ such that $\|\nabla f_{\alpha_k}(\bm{w}^k)\|_2\geq \epsilon$ almost surely for all $k\geq\overline{K}_2$,  
	\begin{alignat}{6}
		\sum_{k=\overline{K}_2}^{\infty}\alpha_k^d\hat{\eta}_k\frac{\Delta_{c_1}}{2}\|\nabla f_{\alpha_k}(\bm{w}^k)\|^2_2\stackrel{\text{a.s.}}{\geq} \frac{\Delta_{c_1}}{2}\epsilon^2\sum_{k=\overline{K}_2}^{\infty}\alpha_k^d\hat{\eta}_k=\infty,\nonumber
	\end{alignat}
	given that $\sum_{k=1}^{\infty}\alpha_k^d\hat{\eta}_k=\infty$ by assumption and  $\sum_{k=1}^{\overline{K}_2-1}\alpha_k^d\hat{\eta}_k$ is finite, contradicting \eqref{bound}. There exists almost surely a subsequence of indices $\{k_i\}$ for which $\lim\limits_{i\rightarrow \infty} \|\nabla f_{\alpha_{k_i}}(\bm{w}^{k_i})\|_2=\liminf\limits_{k\rightarrow \infty} \|\nabla f_{\alpha_{k}}(\bm{w}^{k})\|_2=0$.
	If $\bm{w}^*$ is an accumulation point of $\{\bm{w}^{k_i}\}$, let $\{k_{i_j}\}$ be a subsequence of $\{k_i\}$ such that $\lim\limits_{j\rightarrow \infty} \bm{w}^{k_{i_j}}= \bm{w}^*$. Given that $\partial^{\infty}_{\alpha_{k_{i_j}}} f(\bm{w}^{k_{i_j}})$ converges continuously to $\partial f(\bm{w}^*)$ by Proposition \ref{cont-conv}, it holds that 
	$$\lim\limits_{j\rightarrow \infty}\dist(\bm{0},\partial^{\infty}_{\alpha_{k_{i_j}}} f(\bm{w}^{k_{i_j}}))=\dist(\bm{0},\partial f(\bm{w}^*))$$
	\cite[Exercise 5.42 (b)]{rockafellar2009}. Since $\nabla f_{\alpha_{k_{i_j}}}(\bm{w}^{k_{i_j}})\in \partial^{\infty}_{\alpha_{k_{i_j}}} f(\bm{w}^{k_{i_j}})$ from Proposition \ref{inclusion}, 
	\begin{alignat}{6}
		\dist(\bm{0},\partial f(\bm{w}^*))=\lim\limits_{j\rightarrow \infty}\dist(\bm{0},\partial^{\infty}_{\alpha_{k_{i_j}}} f(\bm{w}^{k_{i_j}}))\leq\lim\limits_{j\rightarrow \infty}\|\nabla f_{\alpha_{k_{i_j}}}(\bm{w}^{k_{i_j}})\|_2=0,\nonumber
	\end{alignat}	
which concludes the proof.
\end{proof}

\noindent The following proposition demonstrates the asymptotic implications of the conditions placed on $\{\Delta_k\}$, $\{\alpha_k\}$, and $\{\hat{\eta}_k\}$.

 \begin{proposition}
 	\label{limits}	
 	Under the conditions of Theorem \ref{convergence}, $\liminf\limits_{k\rightarrow \infty} \Delta_k=0$ and $\liminf\limits_{k\rightarrow \infty} \hat{\eta}_k=0$. 
 \end{proposition}
 \begin{proof}
  	Using the conditions of \eqref{the_cond} and Definition \ref{delta_bound}(1),
 	\begin{alignat}{6}
 	\infty>\sum_{k=1}^{\infty}\alpha_k^{d-1}\hat{\eta}^2_k\geq \sum_{k=1}^{K-1}\alpha_k^{d-1}\hat{\eta}^2_k+ \sum_{k=K}^{\infty}\alpha_k^d\hat{\eta}_k\frac{\Delta_k}{c_4}.\label{prp1}
 	\end{alignat}	
 	
  	\noindent If there exists a $K'\geq K$ and $\epsilon>0$ such that $\Delta_k\geq \epsilon$ for all $k\geq K'$, 
   	\begin{alignat}{6}
  	&\sum_{k=K'}^{\infty}\alpha_k^d\hat{\eta}_k\frac{\Delta_k}{c_4}\geq \frac{\epsilon}{c_4}\sum_{k=K'}^{\infty}\alpha_k^d\hat{\eta}_k=\infty\nonumber\\
  	\Rightarrow&\sum_{k=K}^{\infty}\alpha_k^d\hat{\eta}_k\frac{\Delta_k}{c_4}=\infty\nonumber,  	
  \end{alignat}	
  \noindent contradicting \eqref{prp1}, which proves that $\liminf\limits_{k\rightarrow \infty} \Delta_k=0$. If there exists a $K'\in\NN$ and $\epsilon>0$ such that $\hat{\eta}_k\geq \epsilon$ for all $k\geq K'$, then    
  \begin{alignat}{6}
  	\liminf\limits_{k\rightarrow \infty}\frac{\hat{\eta}_k}{\alpha_k}\geq \liminf\limits_{k\rightarrow \infty}\frac{\epsilon}{\alpha_k}=\infty\nonumber  
  \end{alignat}	
  given that  $\lim\limits_{k\rightarrow \infty}\alpha_k=0$ by assumption, contradicting  $\liminf\limits_{k\rightarrow \infty} \Delta_k=0$. 
 \end{proof}
 
\noindent The next proposition gives a family of sequences $\{\alpha_k\}$ and $\{\hat{\eta}_k\}$ which satisfy the conditions of Theorem \ref{convergence} with $\limsup\limits_{k\rightarrow \infty} \Delta_k>0$. The step sizes $\hat{\eta}_k$ are modelled to have a bounded relative rounding error $\delta_k>-1$ for all $k\in\NN$. Ignoring the rounding error, for all $k\in \hat{\NN}$, $\Delta_k\in (0,2c_1\underline{\Psi}^U)$ is a constant.

\begin{proposition}
	\label{stepsize}	
	Let $u\in(0,1)$, $p=\frac{u}{d+1}$, $q=1-dp$, $v>1$, $c_5>0$, and $\{\delta_k\}\subset[\underline{\delta},\overline{\delta}]$, where $uv>1$ and $-1<\underline{\delta}\leq\overline{\delta}<\infty$. Let $\hat{\NN}:=\{\lceil k^v\rceil: k\in \NN\}$, and for $k\in\NN$ let $\alpha_k=\frac{1}{k^p}$, 
	\begin{alignat}{6}
		\hat{\eta}_k:=\begin{cases}
			\frac{c_5(1+\delta_k)}{k^q}&\text{ if }k\notin \hat{\NN}\\
			\frac{c_5(1+\delta_k)}{k^p}&\text{ if }k\in \hat{\NN},\nonumber
		\end{cases}
	\end{alignat}
	and $\Delta_k=\frac{\hat{\eta}_k}{\alpha_k}$, satisfying Definition \ref{delta_bound}(1) with $c_4=1$. It holds that $\{\alpha_k\}$ is a non-increasing sequence, $\lim\limits_{k\rightarrow \infty}\alpha_k=0$, \eqref{the_cond} holds, and $\limsup\limits_{k\rightarrow \infty}\Delta_k\geq c_5(1+\underline{\delta})$. 
	By setting $c_5<\frac{2c_1\underline{\Psi}^U}{(1+\overline{\delta})}$, it holds that 	
	$\Delta_{c_1}=2c_1\underline{\Psi}^U-c_5(1+\overline{\delta})>0$ satisfies Definition \ref{delta_bound}(2).
\end{proposition}

\begin{proof}
	Setting $\alpha_k=\frac{1}{k^p}$, $\{\alpha_k\}$ is non-increasing with $\lim\limits_{k\rightarrow\infty}\alpha_k=0$ for $p>0$. Since $q=\frac{1+d(1-u)}{d+1}>p$, the summation conditions \eqref{the_cond} hold when     
	\begin{alignat}{6}
		\sum_{k=1}^{\infty}\alpha_k^d\hat{\eta}_k&\geq c_5(1+\underline{\delta})\left(\sum_{k\notin\hat{\NN} }k^{-dp}k^{-q}+\sum_{k\in\hat{\NN}}k^{-dp}k^{-p}\right)\nonumber\\
		&\geq c_5(1+\underline{\delta})\sum_{k=1}^{\infty}k^{-dp}k^{-q}\nonumber\\
		&=\infty,\quad\text{and}\nonumber
	\end{alignat}
	\begin{alignat}{6}
		\sum_{k=1}^{\infty}\alpha_k^{d-1}\hat{\eta}^2_k&\leq c_5^2(1+\overline{\delta})^2 \left(\sum_{k\notin\hat{\NN} }k^{-(d-1)p}k^{-2q}+\sum_{k\in\hat{\NN}}k^{-(d-1)p}k^{-2p}\right)\nonumber\\
		&\leq c_5^2(1+\overline{\delta})^2 \left(\sum_{k=1}^{\infty}k^{-(d-1)p}k^{-2q}+\sum_{k=1}^{\infty}\lceil k^v\rceil^{-(d+1)p}\right)\nonumber\\
		&\leq c_5^2(1+\overline{\delta})^2 \left(\sum_{k=1}^{\infty}k^{-(d-1)p}k^{-2q}+\sum_{k=1}^{\infty}k^{-(d+1)pv}\right)\nonumber\\
		&<\infty,\nonumber
	\end{alignat}
	which is true when $dp+q\leq 1$, $(d-1)p+2q>1$, and $(d+1)pv>1$. Using the conditions of the proposition,  
		\begin{alignat}{6}
		dp+q&=dp+1-dp=1,\nonumber\\
		(d-1)p+2q&=(d-1)p+2-2dp=2-p(d+1)=2-u>1,\text{ and}\nonumber\\
		(d+1)pv&=uv>1,\nonumber
		\end{alignat}
	therefore \eqref{the_cond} is satisfied. 
	
	Given that for all $k\in\hat{\NN}$ $\frac{\hat{\eta}_k}{\alpha_k}=c_5(1+\delta_k)$, $\limsup\limits_{k\rightarrow \infty}\Delta_k\geq c_5(1+\underline{\delta})$ and $\Delta_k\leq c_5(1+\overline{\delta})$ for all $k\in\NN$. Setting $c_5<\frac{2c_1\underline{\Psi}^U}{(1+\overline{\delta})}$,  
	$\Delta_k\leq c_5(1+\overline{\delta})<2c_1\underline{\Psi}^U$,	and by setting $\Delta_{c_1}=2c_1\underline{\Psi}^U-c_5(1+\overline{\delta})>0$, $2c_1\Psi^U_k-\Delta_k\geq \Delta_{c_1}$ for all $k\in\NN$, satisfying Definition \ref{delta_bound}(2). 
\end{proof}

\noindent Giving an asymptotic convergence result in Theorem \ref{convergence} in a setting motivated by finite-precision arithmetic may seem contradictory. It is worth noting that the sequence $\{\alpha_k\}$, the scalar $\Delta_{c_1}$, the function $f_{\alpha}$, and its gradient all live in Euclidean space and never need to be computed. The one actual implementation challenge is the fact that $\liminf\limits_{k\rightarrow \infty}\hat{\eta}_k=0$. If we consider a sequence of fixed-point environments $\{\FF_{t_j}\}$ with increasing fractional digits $t_{j+1}>t_j$ for all $j\in\NN$, a schedule can be followed where $\FF_{t_1}$ is used for iterations $[1,\widehat{K}_1]$, $\FF_{t_2}$ for iterations $[\widehat{K}_1+1,\hat{K}_2]$, and so on for a predetermined sequence $\{\widehat{K}_j\}\subset \NN$, which would accommodate decreasing step sizes. This idea of increasing the number of fractional digits through time was successfully used in \cite[Figure 3]{gupta2015}, where neural network training in an $\FF_{12}$ was performed until stagnation occurred, after which the fractional digits were increased to $t=16$, resulting in a rapid accuracy improvement. A perhaps more practical approach is to consider a fixed $\alpha_k=\alpha>0$, allowing for a non-asymptotic convergence bound in expectation using a fixed $\hat{\eta}_k=\hat{\eta}>0$, which is now shown for the $L_{\infty}$-norm Clarke $\alpha$-subdifferential, where the parameter $c_5>0$ can ensure that for a single fixed $\FF$, $\hat{\eta}\in\FF$. Similar to Proposition \ref{stepsize}, for all $k\in\hat{\NN}$, $\Delta_k\in(0,2c_1\Psi_k^U)$.

\begin{theorem}
	\label{non-asy}
	For any $K\in\NN$ and all $k\in[K]$, let $\alpha_k=\alpha>0$, $\hat{\eta}_k=\hat{\eta}=\frac{c_5}{\sqrt{K}}$ for $c_5>0$, and let $\{\psi_k\}$ satisfy Definition \ref{adaptive_stepsize}. Let $\hat{\NN}:=\{k^2: k\in \NN, k^2\in[K]\}$, $\Delta_{c_1}\in(0,2c_1\underline{\Psi}^U)$, and  
	 \begin{alignat}{6}
	 	\Delta_k:=\begin{cases}
	 		\min(\frac{c_6}{\sqrt{K}},2c_1\Psi^U_k-\Delta_{c_1})&\text{ if }k\in [K]\setminus\hat{\NN}\\
	 		2c_1\Psi^U_k-\Delta_{c_1}&\text{ if }k\in \hat{\NN},\nonumber
	 	\end{cases}
	 \end{alignat}
	 for $c_6>0$ and all $k\in[K]$. If Assumptions \ref{approx_assump} and \ref{e_assump} hold for all $k\in[K]$, PISGD \eqref{eq:1} generates a sequence such that 
	\begin{alignat}{6}
	&&\EE[\dist(\bm{0},\partial^{\infty}_{\alpha}f(\hat{\bm{w}}))^2]&\leq \frac{\kappa_1f_{\alpha}(\bm{w}^1)}{\sqrt{K}}+\frac{\kappa_2Q}{\sqrt{K}}+\frac{\kappa_3\sqrt{d}L_0}{\alpha\sqrt{K}}((\overline{\Psi}^U)^2c_2Q+c_3),\nonumber
	\end{alignat}				
	where $\hat{\bm{w}}\sim U(\{\bm{w}^k\}_{k=1}^K)$ is sampled uniformly over the generated iterates $\{\bm{w}^k\}_{k=1}^K$,  $\kappa_1=\frac{2}{c_5\Delta_{c_1}}$, 	
	$\kappa_2=\frac{c_2}{\Delta_{c_1}}(c_6+2c_1\overline{\Psi}^U-\Delta_{c_1})$, and 
	$\kappa_3=\frac{c_5}{\Delta_{c_1}}$. To guarantee that 
	\begin{alignat}{6}
		&&\EE[\dist(\bm{0},\partial^{\infty}_{\alpha}f(\hat{\bm{w}}))]&\leq \nu\nonumber	
	\end{alignat}
	for any $\nu>0$ requires a $K=O\left(\alpha^{-2}\nu^{-4}\right)$.	
\end{theorem}
\begin{proof}	
	Taking the conditional expectation with respect to $\F_{k-1}$ of inequality \eqref{non-start} in the proof of Theorem \ref{convergence}, and simplifying the notation, letting $\alpha_k=\alpha$ and $\hat{\eta}_k=\hat{\eta}$, 
	\begin{alignat}{6}
		&&&\EE[f_{\alpha}(\bm{w}^{k+1})|\F_{k-1}]\nonumber\\
		&&\leq&f_{\alpha}(\bm{w}^k)-\hat{\eta}\EE[\psi_k \langle \nabla f_{\alpha}(\bm{w}^k),\widehat{\nabla} \overline{F}^k(\bm{w}^k)\rangle|\F_{k-1}]+\langle \nabla f_{\alpha}(\bm{w}^k),\EE[\hat{\bm{e}}^k|\F_{k-1}]\rangle\nonumber\\ 
		&&+&\frac{L_1^{\alpha}}{2}(\hat{\eta}^2\EE[\psi^2_k\|\widehat{\nabla} \overline{F}^k(\bm{w}^k)\|^2_2|\F_{k-1}]-2\hat{\eta}\EE[\psi_k\langle \widehat{\nabla} \overline{F}^k(\bm{w}^k),\hat{\bm{e}}^k\rangle|\F_{k-1}]
		+\EE[\|\hat{\bm{e}}^k\|^2_2|\F_{k-1}])\nonumber\\
		&&\stackrel{\text{a.s.}}{\leq}&f_{\alpha}(\bm{w}^k)-\hat{\eta}\EE[\psi_k \langle \nabla f_{\alpha}(\bm{w}^k),\widehat{\nabla} \overline{F}^k(\bm{w}^k)\rangle|\F_{k-1}]+\hat{\eta}^2\frac{L_1^{\alpha}}{2}((\overline{\Psi}^U)^2c_2Q+c_3)\nonumber\\
		&&\stackrel{\text{a.s.}}{\leq}&f_{\alpha}(\bm{w}^k)-\hat{\eta}
		\left(c_1\Psi^U_k-\frac{\Delta_k}{2}\right)\|\nabla f_{\alpha}(\bm{w}^k)\|^2_2	
		+\hat{\eta}\frac{\Delta_k}{2}c_2Q+\hat{\eta}^2\frac{L_1^{\alpha}}{2}((\overline{\Psi}^U)^2c_2Q+c_3)\nonumber\\
		&&\leq&f_{\alpha}(\bm{w}^k)-\hat{\eta}
		\frac{\Delta_{c_1}}{2}\|\nabla f_{\alpha}(\bm{w}^k)\|^2_2	
		+\hat{\eta}\frac{\Delta_k}{2}c_2Q+\hat{\eta}^2\frac{L_1^{\alpha}}{2}((\overline{\Psi}^U)^2c_2Q+c_3),\nonumber
	\end{alignat} 
	where the second inequality holds using the same simplifications used to get inequality \eqref{prerobbsieg2}, and the third inequality uses \eqref{pre_cross_term}, both in the proof of Theorem \ref{convergence}. The fourth inequality uses the fact that
	$\Delta_k\leq 2c_1\Psi^U_k-\Delta_{c_1}$. Dividing by $\hat{\eta}
	\frac{\Delta_{c_1}}{2}$ and rearranging,
	\begin{alignat}{6}
		&&&\|\nabla f_{\alpha}(\bm{w}^k)\|^2_2\nonumber\\
		&&\stackrel{\text{a.s.}}{\leq}&\frac{2}{\hat{\eta}\Delta_{c_1}}		(f_{\alpha}(\bm{w}^k)-\EE[f_{\alpha}(\bm{w}^{k+1})|\F_{k-1}])+\frac{\Delta_k}{\Delta_{c_1}}c_2Q+\hat\eta\frac{L_1^{\alpha}}{\Delta_{c_1}}((\overline{\Psi}^U)^2c_2Q+c_3).\nonumber
	\end{alignat} 
	Taking the expectation, summing the inequalities over $k\in[K]$, and dividing by $K$, 	
	\begin{alignat}{6}
		&&\frac{1}{K}\sum_{k=1}^K\EE[\|\nabla f_{\alpha}(\bm{w}^k)\|^2_2]&\leq \frac{2}{\hat{\eta}\Delta_{c_1}K}		(f_{\alpha}(\bm{w}^1)-\EE[f_{\alpha}(\bm{w}^{K+1})])+\frac{c_2Q}{\Delta_{c_1}}\frac{1}{K}\sum_{k=1}^K\EE[\Delta_k]\nonumber\\
		&&&+\hat\eta\frac{L_1^{\alpha}}{\Delta_{c_1}}((\overline{\Psi}^U)^2c_2Q+c_3).\label{finite_1}
	\end{alignat}
	Analyzing $\frac{1}{K}\sum_{k=1}^K\EE[\Delta_k]=\frac{1}{K}\EE[\sum_{k=1}^K\Delta_k]$,
	\begin{alignat}{6}
	\sum_{k=1}^K\Delta_k&=\sum_{k\in [K]\setminus\hat{\NN}}\min\left(\frac{c_6}{\sqrt{K}},2c_1\Psi^U_k-\Delta_{c_1}\right)+\sum_{k\in \hat{\NN}}(2c_1\Psi^U_k-\Delta_{c_1})\nonumber\\
	&\leq\sum_{k\in [K]\setminus\hat{\NN}}\frac{c_6}{\sqrt{K}}+\sum_{k\in \hat{\NN}}(2c_1\Psi^U_k-\Delta_{c_1})\nonumber\\
	&\stackrel{\text{a.s.}}{\leq}\frac{c_6(K-\lfloor\sqrt{K}\rfloor)}{\sqrt{K}}+(2c_1\overline{\Psi}^U-\Delta_{c_1})\lfloor\sqrt{K}\rfloor\nonumber\\
	&\leq(c_6+2c_1\overline{\Psi}^U-\Delta_{c_1})\sqrt{K}\nonumber\\
	\Rightarrow\frac{1}{K}\sum_{k=1}^K\EE[\Delta_k]&\leq\frac{(c_6+2c_1\overline{\Psi}^U-\Delta_{c_1})}{\sqrt{K}}.\nonumber
	\end{alignat}

	Plugging this bound into \eqref{finite_1}, with the definitions of $\hat{\eta}$ and $L_1^{\alpha}$,
	\begin{alignat}{6}
	&&&\frac{1}{K}\sum_{k=1}^K\EE[\|\nabla f_{\alpha}(\bm{w}^k)\|^2_2]\nonumber\\
	&&\leq&\frac{2}{c_5\Delta_{c_1}\sqrt{K}}		(f_{\alpha}(\bm{w}^1)-\EE[f_{\alpha}(\bm{w}^{K+1})])
	+\frac{c_2Q}{\Delta_{c_1}\sqrt{K}}(c_6+2c_1\overline{\Psi}^U-\Delta_{c_1})\nonumber\\
	&&+&\frac{c_5\sqrt{d}L_0}{\alpha\Delta_{c_1}\sqrt{K}}((\overline{\Psi}^U)^2c_2Q+c_3)\nonumber\\
	&&=&\frac{\kappa_1}{\sqrt{K}}(f_{\alpha}(\bm{w}^1)-\EE[f_{\alpha}(\bm{w}^{K+1})])
	+\frac{\kappa_2Q}{\sqrt{K}}+\frac{\kappa_3\sqrt{d}L_0}{\alpha\sqrt{K}}((\overline{\Psi}^U)^2c_2Q+c_3),\nonumber
	\end{alignat}	
	using $\kappa_1=\frac{2}{c_5\Delta_{c_1}}$, 	
	$\kappa_2=\frac{c_2}{\Delta_{c_1}}(c_6+2c_1\overline{\Psi}^U-\Delta_{c_1})$, and 
	$\kappa_3=\frac{c_5}{\Delta_{c_1}}$. 
	
	Noting that $\EE[\|\nabla f_{\alpha}(\hat{\bm{w}})\|^2_2]=\frac{1}{K}\sum_{k=1}^K\EE[\|\nabla f_{\alpha}(\bm{w}^k)\|^2_2]$, $\dist(\bm{0},\partial^{\infty}_{\alpha}f(\hat{\bm{w}}))\leq \|\nabla f_{\alpha}(\hat{\bm{w}})\|_2$ from Proposition \ref{inclusion}, and that $\EE[f_{\alpha}(\bm{w}^{K+1})]\geq 0$ by Assumption \ref{pos_loss}, 
	\begin{alignat}{6}
		&&\EE[\dist(\bm{0},\partial^{\infty}_{\alpha}f(\hat{\bm{w}}))^2]&\leq \EE[\|\nabla f_{\alpha}(\hat{\bm{w}})\|^2_2]\nonumber\\
		&&&\leq \frac{\kappa_1f_{\alpha}(\bm{w}^1)}{\sqrt{K}}+\frac{\kappa_2Q}{\sqrt{K}}+\frac{\kappa_3\sqrt{d}L_0}{\alpha\sqrt{K}}((\overline{\Psi}^U)^2c_2Q+c_3).\nonumber
	\end{alignat}	
	Given that $\EE[\dist(\bm{0},\partial^{\infty}_{\alpha}f(\hat{\bm{w}}))]^2\leq\EE[\dist(\bm{0},\partial^{\infty}_{\alpha}f(\hat{\bm{w}}))^2]$ by Jensen's inequality, the requirement that 
	$\EE[\dist(\bm{0},\partial^{\infty}_{\alpha}f(\hat{\bm{w}}))]\leq \nu$ is satisfied when   
	\begin{alignat}{6}
		&&\frac{\kappa_1f_{\alpha}(\bm{w}^1)}{\sqrt{K}}+\frac{\kappa_2Q}{\sqrt{K}}+\frac{\kappa_3\sqrt{d}L_0}{\alpha\sqrt{K}}((\overline{\Psi}^U)^2c_2Q+c_3)&\leq\nu^2\nonumber\\
		\Rightarrow&&\frac{1}{\alpha^2\nu^4}\left(\alpha\kappa_1 f_{\alpha}(\bm{w}^1)+\alpha\kappa_2Q+\kappa_3\sqrt{d}L_0((\overline{\Psi}^U)^2c_2Q+c_3)\right)^2&\leq K,\nonumber
	\end{alignat}
	proving that $\EE[\dist(\bm{0},\partial^{\infty}_{\alpha}f(\hat{\bm{w}}))]\leq \nu$ is guaranteed for a $K=O\left(\alpha^{-2}\nu^{-4}\right)$.
\end{proof}

\section{Convergence Analysis with Unrestricted Adaptive Step Size Range}
\label{uncon_conv}
In the previous section, the convergence results hold for adaptive step sizes with $\limsup\limits_{k\rightarrow \infty}\Delta_k>0$, but the density of $k\in\NN$ with a fixed or lower-bounded $\Delta_k>0$ equals $0$. An obstacle preventing the use of more general step sizes is the lack of variance reduction from using a mini-batch size $M>1$. This stems from the bias of the stochastic gradient estimate $\widehat{\nabla} \overline{F}^k(\bm{w}^k)$. By replacing $\widehat{\nabla} F(\bm{w}^k+\hat{\bm{u}}^k,\bm{\xi}^{k,i},\bm{b}^{k,i})$ with the infinite precision stochastic gradient $\widetilde{\nabla} F(\bm{w}^k+\bm{u}^k,\bm{\xi}^{k,i})$, non-asymptotic convergence for all adaptive step sizes satisfying Definition \ref{adapt_class} with $\Psi^L_k>0$ can be proven, precisely because $\EE[\|\frac{1}{M}\sum_{i=1}^M\widetilde{\nabla} F(\bm{w}^k+\bm{u}^k,\bm{\xi}^{k,i})-\nabla f_{\alpha_k}(\bm{w}^k)\|^2_2|\F_{k-1}]\leq \frac{Q}{M}$ \cite[Proposition 7(2)]{metel2022}. This same type of variance reduction also holds under the standard assumptions of non-convex stochastic optimization, where for an unbiased stochastic gradient $\bm{g}$ of the loss function's gradient $\nabla f(\bm{w})$, it is assumed that $\EE[\|\bm{g}-\nabla f(\bm{w})\|^2_2]\leq Q$ \cite[Assumption A1 \& Inequality (4.12)]{GhadimiMinibatch}.

In this section Assumption \ref{approx_assump} will be relaxed, balancing the initial motivation of optimization with numerical error, with the potential general interest of this section's convergence analysis.  By taking $\widehat{\nabla} \overline{F}^k(\bm{w}^k)=\frac{1}{M}\sum_{i=1}^M\widetilde{\nabla} F(\bm{w}^k+\bm{u}^k,\bm{\xi}^{k,i})$, the following assumption is satisfied with $c'_1=0$ and $c'_2=1$.

 \begin{assumption}\label{approx_assump_relaxed}
 	There exist constants $c'_1>0$ and $c'_2>0$ such that for all $k\in\NN$, almost surely,    
 	\begin{alignat}{6}
 		\langle \EE[\widehat{\nabla} \overline{F}^k(\bm{w}^k)|\F_{k-1}],\nabla f_{\alpha_k}(\bm{w}^k)\rangle&\geq \left(1-\frac{c'_1}{M}\right)\|\nabla f_{\alpha_k}(\bm{w}^k)\|^2_2\quad\text{and}\nonumber\\	
 		\EE[\|\widehat{\nabla} \overline{F}^k(\bm{w}^k)\|^2_2|\F_{k-1}]&\leq \|\nabla f_{\alpha_k}(\bm{w}^k)\|^2_2+\frac{c'_2Q}{M}.\nonumber
 	\end{alignat}
 \end{assumption}
 
 \noindent Compared to Assumption \ref{approx_assump}, a weak descent condition still exists, and the second moment of $\|\widehat{\nabla} \overline{F}^k(\bm{w}^k)\|_2$ is greater than the squared $L_2$-norm of the true gradient, but both conditions improve with increasing $M$. With these assumptions, the desired variance reduction is obtained.   
 \begin{proposition}
 	\label{varbound}
 	If Assumption \ref{approx_assump_relaxed} holds, for all $k\in\NN$ almost surely, 	
 	$$\EE[\|\widehat{\nabla} \overline{F}^k(\bm{w}^k)-\nabla f_{\alpha_k}(\bm{w}^k)\|^2_2|\F_{k-1}]\leq \frac{2c'_1}{M}\|\nabla f_{\alpha_k}(\bm{w}^k)\|^2_2+\frac{c'_2Q}{M}.$$
 \end{proposition}
 \begin{proof}
 		\begin{alignat}{6}
 		&\EE[\|\widehat{\nabla} \overline{F}^k(\bm{w}^k)-\nabla f_{\alpha_k}(\bm{w}^k)\|^2_2|\F_{k-1}]\nonumber\\
 		=&\EE[\|\widehat{\nabla} \overline{F}^k(\bm{w}^k)\|^2_2]-2\EE[\langle \widehat{\nabla} \overline{F}^k(\bm{w}^k),\nabla f_{\alpha_k}(\bm{w}^k)\rangle]+\|\nabla f_{\alpha_k}(\bm{w}^k)\|^2_2\nonumber\\
 		\stackrel{\text{a.s.}}{\leq}&\|\nabla f_{\alpha_k}(\bm{w}^k)\|^2_2+\frac{c'_2Q}{M}-2\left(1-\frac{c'_1}{M}\right)\|\nabla f_{\alpha_k}(\bm{w}^k)\|^2_2+\|\nabla f_{\alpha_k}(\bm{w}^k)\|^2_2\nonumber\\
 	 	=&\frac{2c'_1}{M}\|\nabla f_{\alpha_k}(\bm{w}^k)\|^2_2+\frac{c'_2Q}{M}.\nonumber
 	\end{alignat}
 \end{proof}

\noindent A non-asymptotic convergence result similar to Theorem \ref{non-asy} is now proven for all adaptive step sizes satisfying Definition \ref{adapt_class} with $\Psi^L_k>0$ almost surely for all $k\in\NN$, ensuring the strict positivity of step sizes.

\begin{theorem}
	\label{non-asy-fixed}
	For a $K\in\NN$ and all $k\in[K]$, let $\alpha_k=\alpha>0$ and  $\hat{\eta}_k=\hat{\eta}=\frac{c_5}{\sqrt{K}}$ for $c_5>0$. Let $\{\psi_k\}$ satisfy Definition \ref{adaptive_stepsize}, where almost surely $\Psi^L_k\geq\underline{\Psi}^L>0$, $\Delta_L\in(0,\frac{\underline{\Psi}^L}{2})$, and $M=\lceil \sqrt{K}\rceil$. If Assumptions \ref{e_assump} and \ref{approx_assump_relaxed} hold for all $k\in[K]$, and  
 	\begin{alignat}{6}
		K\geq\min\left(\left(\frac{\overline{\Psi}^Uc'_1}{\frac{\underline{\Psi}^L}{2}-\Delta_L}\right)^2,\left(\frac{c_5(\overline{\Psi}^U)^2\sqrt{d}L_0}{\alpha\underline{\Psi}^L}\right)^2\right),\label{iter_bound}
	\end{alignat}	
	PISGD \eqref{eq:1} generates a sequence such that 
	\begin{alignat}{6}
		&&\EE[\dist(\bm{0},\partial^{\infty}_{\alpha}f(\hat{\bm{w}}))^2]&\leq \frac{\kappa_1f_{\alpha}(\bm{w}^1)}{\sqrt{K}}+\frac{\kappa_2Q}{\sqrt{K}}+\frac{\kappa_3\sqrt{d}L_0}{\alpha\sqrt{K}},\nonumber
	\end{alignat}				
	where $\hat{\bm{w}}\sim U(\{\bm{w}^k\}_{k=1}^K)$ is sampled uniformly over the generated iterates $\{\bm{w}^k\}_{k=1}^K$,   $\kappa_1=\frac{1}{c_5\Delta_{L}}$, $\kappa_2=\frac{\overline{\Psi}^Uc'_2}{2\Delta_L}$, and 
	$\kappa_3=\frac{c_3c_5}{2\Delta_L}$. To guarantee that 
	\begin{alignat}{6}
		&&\EE[\dist(\bm{0},\partial^{\infty}_{\alpha}f(\hat{\bm{w}}))]&\leq \nu\nonumber	
	\end{alignat}
	for any $\nu>0$ requires $K=O\left(\alpha^{-2}\nu^{-4}\right)$ iterations and $O\left(\alpha^{-3}\nu^{-6}\right)$  $\widehat{\nabla} F$ stochastic gradient calls.	
\end{theorem}
\begin{proof}	
	Taking the conditional expectation with respect to $\F_{k-1}$ of inequality \eqref{non-start} in the proof of Theorem \ref{convergence}, and simplifying the notation, letting $\alpha_k=\alpha$ and $\hat{\eta}_k=\hat{\eta}$, 
	\begin{alignat}{6}
		&&&\EE[f_{\alpha}(\bm{w}^{k+1})|\F_{k-1}]\nonumber\\
		&&\leq&f_{\alpha}(\bm{w}^k)-\hat{\eta}\EE[\psi_k \langle \nabla f_{\alpha}(\bm{w}^k),\widehat{\nabla} \overline{F}^k(\bm{w}^k)\rangle|\F_{k-1}]+\langle \nabla f_{\alpha}(\bm{w}^k),\EE[\hat{\bm{e}}^k|\F_{k-1}]\rangle\nonumber\\ 
		&&+&\frac{L_1^{\alpha}}{2}(\hat{\eta}^2\EE[\psi^2_k\|\widehat{\nabla} \overline{F}^k(\bm{w}^k)\|^2_2|\F_{k-1}]-2\hat{\eta}\EE[\psi_k\langle \widehat{\nabla} \overline{F}^k(\bm{w}^k),\hat{\bm{e}}^k\rangle|\F_{k-1}]
		+\EE[\|\hat{\bm{e}}^k\|^2_2|\F_{k-1}])\nonumber\\
		&&\stackrel{\text{a.s.}}{\leq}&f_{\alpha}(\bm{w}^k)-\hat{\eta}\EE[\psi_k \langle \nabla f_{\alpha}(\bm{w}^k),\widehat{\nabla} \overline{F}^k(\bm{w}^k)\rangle|\F_{k-1}]\nonumber\\
		&&+&\hat{\eta}^2\frac{L_1^{\alpha}}{2}((\Psi^U_k)^2\EE[\|\widehat{\nabla} \overline{F}^k(\bm{w}^k)\|^2_2|\F_{k-1}]+c_3)\nonumber\\
		&&\stackrel{\text{a.s.}}{\leq}&f_{\alpha}(\bm{w}^k)
		+\hat{\eta}\left(\frac{\Psi^U_k}{2}\EE[\|\nabla f_{\alpha_k}(\bm{w}^k)-\widehat{\nabla} \overline{F}^k(\bm{w}^k)\|^2_2|\F_{k-1}]-\frac{\Psi^L_k}{2}\|\nabla f_{\alpha_k}(\bm{w}^k)\|^2_2\right.\nonumber\\	
		&&-&\left.\frac{\Psi^L_k}{2}\EE[\|\widehat{\nabla} \overline{F}^k(\bm{w}^k)\|^2_2|\F_{k-1}]\right)+\hat{\eta}^2\frac{L_1^{\alpha}}{2}((\Psi^U_k)^2\EE[\|\widehat{\nabla} \overline{F}^k(\bm{w}^k)\|^2_2|\F_{k-1}]+c_3)\nonumber\\	
		&&\stackrel{\text{a.s.}}{\leq}&f_{\alpha}(\bm{w}^k)
		+\hat{\eta}\left(\frac{\Psi^U_k}{2}\left(\frac{2c'_1}{M}\|\nabla f_{\alpha_k}(\bm{w}^k)\|^2_2+\frac{c'_2Q}{M}\right)-\frac{\Psi^L_k}{2}\|\nabla f_{\alpha_k}(\bm{w}^k)\|^2_2\right.\nonumber\\	
		&&-&\left.\frac{\Psi^L_k}{2}\EE[\|\widehat{\nabla} \overline{F}^k(\bm{w}^k)\|^2_2|\F_{k-1}]\right)+\hat{\eta}^2\frac{L_1^{\alpha}}{2}((\Psi^U_k)^2\EE[\|\widehat{\nabla} \overline{F}^k(\bm{w}^k)\|^2_2|\F_{k-1}]+c_3),\label{eq_cut2}
	\end{alignat}		
	where the second inequality holds using the same simplifications used to get inequality \eqref{prerobbsieg2},
	except that $\EE[\psi^2_k\|\widehat{\nabla} \overline{F}^k(\bm{w}^k)\|^2_2|\F_{k-1}]$ is simply bounded by $(\Psi^U_k)^2\EE[\|\widehat{\nabla} \overline{F}^k(\bm{w}^k)\|^2_2|\F_{k-1}]$. The third and fourth inequalities use \eqref{var_reduce} and Proposition \ref{varbound}, respectively. Continuing from \eqref{eq_cut2},
	\begin{alignat}{6}		
		\EE[f_{\alpha}(\bm{w}^{k+1})|\F_{k-1}]&&\stackrel{\text{a.s.}}{\leq}&f_{\alpha}(\bm{w}^k)
		+\hat{\eta}\left(\Psi^U_k\frac{c'_1}{M}-\frac{\Psi^L_k}{2}\right)\|\nabla f_{\alpha_k}(\bm{w}^k)\|^2_2+\hat{\eta}\frac{\overline{\Psi}^U}{2}\frac{c'_2Q}{M}\nonumber\\	
		&&+&\frac{\hat{\eta}}{2}\left(\hat{\eta}L_1^{\alpha}(\Psi^U_k)^2-\Psi^L_k\right)\EE[\|\widehat{\nabla} \overline{F}^k(\bm{w}^k)\|^2_2|\F_{k-1}]+\hat{\eta}^2\frac{L_1^{\alpha}}{2}c_3\nonumber\\
		&&\stackrel{\text{a.s.}}{\leq}&f_{\alpha}(\bm{w}^k)
		-\hat{\eta}\Delta_L\|\nabla f_{\alpha_k}(\bm{w}^k)\|^2_2+\hat{\eta}\frac{\overline{\Psi}^U}{2}\frac{c'_2Q}{M}+\hat{\eta}^2\frac{L_1^{\alpha}}{2}c_3,\label{phase_1}
	\end{alignat} 
	where the two simplifications in the last inequality hold given \eqref{iter_bound}: 
	\begin{alignat}{6}
	\Psi^U_k\frac{c'_1}{M}-\frac{\Psi^L_k}{2}&\stackrel{\text{a.s.}}{\leq} \overline{\Psi}^U\frac{c'_1}{\sqrt{K}}-\frac{\underline{\Psi}^L}{2}\nonumber\\
	&\leq \overline{\Psi}^Uc'_1\left(\frac{\frac{\underline{\Psi}^L}{2}-\Delta_L}{\overline{\Psi}^Uc'_1}\right)-\frac{\underline{\Psi}^L}{2}\nonumber\\		
	&=-\Delta_L,\quad\text{and}\nonumber\\
	\hat{\eta}L_1^{\alpha}(\Psi^U_k)^2-\Psi^L_k&\stackrel{\text{a.s.}}{\leq}\frac{c_5}{\sqrt{K}}\frac{\sqrt{d}L_0}{\alpha}(\overline{\Psi}^U)^2-\underline{\Psi}^L\nonumber\\
	&\leq c_5\frac{\alpha\underline{\Psi}^L}{c_5(\overline{\Psi}^U)^2\sqrt{d}L_0}\frac{\sqrt{d}L_0}{\alpha}(\overline{\Psi}^U)^2-\underline{\Psi}^L\nonumber\\ 		
	&=0.\nonumber
	\end{alignat}		
	Dividing \eqref{phase_1} by $\hat{\eta}\Delta_L$ and rearranging,
	\begin{alignat}{6}
		\|\nabla f_{\alpha}(\bm{w}^k)\|^2_2\stackrel{\text{a.s.}}{\leq}&\frac{1}{\hat{\eta}\Delta_{L}}		(f_{\alpha}(\bm{w}^k)-\EE[f_{\alpha}(\bm{w}^{k+1})|\F_{k-1}])+\frac{\overline{\Psi}^Uc'_2Q}{2M\Delta_L}+\hat{\eta}\frac{L_1^{\alpha}c_3}{2\Delta_L}.\nonumber
	\end{alignat} 
	Taking the expectation, summing the inequalities over $k\in[K]$, and dividing by $K$, 	
	\begin{alignat}{6}
		&&\frac{1}{K}\sum_{k=1}^K\EE[\|\nabla f_{\alpha}(\bm{w}^k)\|^2_2]&\leq\frac{1}{\hat{\eta}\Delta_{L}K}		(f_{\alpha}(\bm{w}^1)-\EE[f_{\alpha}(\bm{w}^{K+1})])+\frac{\overline{\Psi}^Uc'_2Q}{2M\Delta_L}+\hat{\eta}\frac{L_1^{\alpha}c_3}{2\Delta_L}\nonumber\\
		&&&\leq\frac{f_{\alpha}(\bm{w}^1)-\EE[f_{\alpha}(\bm{w}^{K+1})]}{c_5\Delta_{L}\sqrt{K}}		+\frac{\overline{\Psi}^Uc'_2Q}{2\sqrt{K}\Delta_L}+\frac{c_5\sqrt{d}L_0c_3}{2\alpha\sqrt{K}\Delta_L}\nonumber\\		
		&&&=\frac{\kappa_1(f_{\alpha}(\bm{w}^1)-\EE[f_{\alpha}(\bm{w}^{K+1})])}{\sqrt{K}}		+\frac{\kappa_2Q}{\sqrt{K}}+\frac{\kappa_3\sqrt{d}L_0}{\alpha\sqrt{K}},\nonumber		
	\end{alignat}
	using $\kappa_1=\frac{1}{c_5\Delta_{L}}$, 	
	$\kappa_2=\frac{\overline{\Psi}^Uc'_2}{2\Delta_L}$, and 
	$\kappa_3=\frac{c_3c_5}{2\Delta_L}$. As shown in the proof of Theorem \ref{non-asy}, it follows that 
	\begin{alignat}{6}
		&&\EE[\dist(\bm{0},\partial^{\infty}_{\alpha}f(\hat{\bm{w}}))^2]&\leq \frac{\kappa_1f_{\alpha}(\bm{w}^1)}{\sqrt{K}}+\frac{\kappa_2Q}{\sqrt{K}}+\frac{\kappa_3\sqrt{d}L_0}{\alpha\sqrt{K}},\nonumber
	\end{alignat}	
	and the requirement that $\EE[\dist(\bm{0},\partial^{\infty}_{\alpha}f(\hat{\bm{w}}))]\leq \nu$ is satisfied for a $K=O\left(\alpha^{-2}\nu^{-4}\right)$. Given that for each iteration $M=\lceil \sqrt{K}\rceil$ calls to $\widehat{\nabla} F$ are made, the total number of stochastic gradient calls required equals $K\lceil \sqrt{K}\rceil=O\left(\alpha^{-3}\nu^{-6}\right)$.
\end{proof}

\section{Numerical Demonstration of an Adaptive Step Size \& Empirical Verification of Assumption \ref{approx_assump}}
\label{empirical}

In this section we first develop and test an adaptive step size satisfying Definition \ref{adaptive_stepsize} for fixed-point arithmetic environments. Two ResNet models are trained: ResNet 20 on CIFAR-10 (R20C10)$\label{r20}$ and ResNet 32 on CIFAR-100 (R32C100)$\label{r32}$. The experiments were conducted using QPyTorch \citep{zhang2019}, which enabled the simulation of training using fixed-point arithmetic with stochastic rounding, which is the rounding method of choice for lower-precision deep learning \citep{gupta2015,wang2022,yang2019}. 

\subsection{Restricted Gradient Normalization}
\label{GN}

As an example from the class of adaptive step sizes proposed in Section \ref{adapt_class}, Restricted Gradient Normalization (RGN) is presented in Algorithm \ref{alg1}.\footnote{When $k=1$, $\eta_1$ is set to $\hat{\eta_1}$.} To motivate this step size, and by extension the generality offered by Definition \ref{adaptive_stepsize}, the more common form of normalized SGD is first considered, $\eta_k=\frac{\hat{\eta}_k}{g^k_{nrm}}$
[\citealp[Equation 2.7]{shor1998}, \citealp[Section 3.2.3]{nesterov2004}], where $\hat{\eta}_k$ is a deterministic step size and  $\psi_k=(g^k_{nrm})^{-1}$, following the notation of Section \ref{adapt_class}. 

Given that it is unclear how to choose $\hat{\eta}_k$, the quantity $\psi'_k:=\frac{m^{k-1}_{ave}}{g^{k}_{nrm}}$ is used in RGN, which is the intended value for $\psi_k$ before satisfying the conditions of Definition \ref{adaptive_stepsize} and taking into account rounding error. The denominator $g^{k}_{nrm}:=\max(R(\|\widehat{\nabla} \overline{F}(\bm{w}^k)\|_1),\mu)$ is approximately equal to the norm of $\widehat{\nabla} \overline{F}(\bm{w}^k)$, where $\mu>0$ is a small positive constant to avoid division by $0$. The numerator, $m^{k-1}_{ave}:=R\left(\frac{1}{\min(k-1,c)}\sum_{i=\max(1,k-c)}^{k-1}g^i_{nrm}\right)$, is the average of past values of $g^{k}_{nrm}$, where $c=10$ was used for all experiments.
	
If the norm of the gradient is larger (smaller) than the moving average, the step size decreases (increases), with the intention of stabilizing the norm of the algorithm's updates $\|w^{k+1}-w^k\|_2$ through time. Assuming that $\EE[\psi_k]\approx 1$, the need to tune $\{\hat{\eta}_k\}$ can be avoided by setting it equal to the step size used with SGD, allowing for a clear comparison between (PI)SGD with and without RGN. 

The quantity $m^{k-1}_{ave}/g^{k-1}_{nrm}$ is used to construct $\F_{k-1}$-measurable bounds $\Psi^L_k\geq 0$ and $\Psi^U_k>0$ to clip $\psi'_k$ if needed. Assuming that $\psi'_k$ is unimodal and symmetric about $m^{k-1}_{ave}/g^{k-1}_{nrm}$, the values of $\Psi^L_k$ and $\Psi^U_k$ are dynamically chosen as evenly and as far apart as possible from $m^{k-1}_{ave}/g^{k-1}_{nrm}$, minimizing the probability of clipping $\psi'_k$. 

Higher accuracy in experiments was found by using the L1-norm when computing $g^{k}_{nrm}$ and a simple moving average when computing $m^{k-1}_{ave}$ compared to using the L2-norm and an exponential moving average with a weight parameter equal to $R(0.1)$. Our reasoning for this is that computationally simpler operations are in general less negatively affected by rounding error.

\begin{algorithm}
	\caption{RGN: Restricted Gradient Normalization for $k>1$} 
	\begin{algorithmic}
		\STATE {\bfseries Input:} $\widehat{\nabla} \overline{F}(\bm{w}^k)\in \FF^d$; $\{g^{i}_{nrm}\}_{i=\max(1,k-c)}^{k-1}\subset \FF_{>0}$; $\hat{\eta}_k,\mu\in \FF_{>0}$; $\frac{\Delta_k}{2}\in \RR_{\geq 0}$	
		\STATE $g^k_{nrm}=\max(R(\|\widehat{\nabla} \overline{F}(\bm{w}^k)\|_1),\mu)$
		\STATE $m^{k-1}_{ave}=R\left(\frac{1}{\min(k-1,c)}\sum_{i=\max(1,k-c)}^{k-1}g^i_{nrm}\right)$	
		\STATE $\upsilon_k=\min\left(\frac{\Delta_k}{2},\frac{m^{k-1}_{ave}}{g^{k-1}_{nrm}}\right)$
		\STATE $\Psi^L_k=\frac{m^{k-1}_{ave}}{g^{k-1}_{nrm}}-\upsilon_k$
		\STATE $\Psi^U_k=\frac{m^{k-1}_{ave}}{g^{k-1}_{nrm}}+\Delta_k-\upsilon_k$
		\STATE $\psi_k=\min\left(\max\left(\Psi^L_k,\frac{m^{k-1}_{ave}}{g^{k}_{nrm}}\right),\Psi^U_k\right)$
		\STATE {\bfseries Output:} $R(\hat{\eta}_k*\psi_k)$
	\end{algorithmic}
	\label{alg1}
\end{algorithm} 

\subsection{Stabilizing Training in Fixed-Point Environments}
\label{exp1}

For all experiments training was done for 200 epochs, with an initial step size of $\hat{\eta}_k=0.1$ which was divided by 10 after 100 epochs, using a mini-batch size of $M=128$, following the original ResNet paper and what is used in practice \cite{he2016,idel18}.\footnote{Dividing the step size again at the $150^{th}$ epoch had an unobservable effect.}  

We do not tune $\{\Delta_k\}$, but instead simply use Algorithm \ref{alg1} with $\Delta_k\geq 2\Lambda^+/\lambda$, which results in no clipping occurring when computing $\hat{\psi}_k$.\footnote{Given that $m^{k-1}_{ave}, g^{k-1}_{nrm}\in \FF_{>0}$, $\frac{m^{k-1}_{ave}}{g^{k-1}_{nrm}}\leq \Lambda^+/\lambda$.} This step size is referred to as our version of Gradient Normalization$\label{N_Grad}$ (GN), with rounded step size  $R(\eta_k)=R\left(\hat{\eta}_k*\frac{m^{k-1}_{ave}}{g^{k}_{nrm}}\right)$. In the implementation of GN, only three rounding operations are performed to compute $g^k_{nrm}$, $m^{k-1}_{ave}$, and $R(\eta_k)$. This implicitly assumes that intermediate steps are stored in sufficiently high precision such that no additional rounding error is observable in the final output. This choice is consistent with the implementation of rounding using QPyTorch, where a quantization layer is added after each neural network layer. 

\begin{figure}[t]
	\centering
	\pgfplotsset{width=12cm,height=3.5cm,compat=1.3}
	\begin{tikzpicture}
		\pgfplotsset{scale only axis}
		\begin{axis}[
			axis y line*=left,axis x line*=bottom,xmin=10,xmax=201,ymin=0.82,ymax=0.88,
			xtick={0,20,...,200},
			ytick={0,0.02,...,0.9},
			xlabel=epochs,
			ylabel=test set accuracy,
			y label style={at={(axis description cs:-0.065,.5)},anchor=south},
			x label style={at={(axis description cs:.5,-0.3)},anchor=south}]
			\addplot[line width=2pt,draw=black]
			table[x=x,y=y-test]
			{FI-CF10-sto-FF.dat};\label{plot_five}		
			\addplot[line width=2pt,draw=cyan]
			table[x=x,y=y-test]
			{FI-CF10-sto-TF.dat};\label{plot_seven}
			\addplot[line width=2pt,draw=orange]
			table[x=x,y=y-test]
			{FI-CF10-sto-TT.dat};\label{plot_twelve}
			\addplot[line width=1pt,draw=black]
			table[x=x,y=y-test]
			{FI-CF10-sto-FF-min.dat};		
			\addplot[line width=1pt,draw=black,dotted]
			table[x=x,y=y-test]
			{FI-CF10-sto-FF-max.dat};	
			\addplot[line width=1pt,draw=cyan]
			table[x=x,y=y-test]
			{FI-CF10-sto-TF-min.dat};
			\addplot[line width=1pt,draw=orange]
			table[x=x,y=y-test]
			{FI-CF10-sto-TT-min.dat};
			\addplot[line width=1pt,draw=cyan,dotted]
			table[x=x,y=y-test]
			{FI-CF10-sto-TF-max.dat};
			\addplot[line width=1pt,draw=orange,dotted]
			table[x=x,y=y-test]
			{FI-CF10-sto-TT-max.dat};			
		\end{axis}
		\matrix[matrix of nodes,anchor=west,xshift=2.7cm,yshift=-1.5cm,inner sep=0.2em,draw,
		column 2/.style={anchor=base west},
		column 4/.style={anchor=base west}]
		{\ref{plot_five}&SGD&[1pt]&\ref{plot_seven}&PISGD&[1pt]&\ref{plot_twelve}&PNSGD\\};
	\end{tikzpicture}
	\begin{tikzpicture}
		\pgfplotsset{scale only axis}
		\begin{axis}[
			axis y line*=left,axis x line*=bottom,xmin=10,xmax=201,ymin=0.55,ymax=0.63,
			xtick={0,20,...,200},
			ytick={0.01,0.03,...,0.65},
			xlabel=epochs,
			ylabel=test set accuracy,
			y label style={at={(axis description cs:-0.065,.5)},anchor=south},
			x label style={at={(axis description cs:.5,-0.3)},anchor=south}]
			\addplot[line width=2pt,draw=black]
			table[x=x,y=y-test]
			{FI-CF100-sto-FF.dat};
			\addplot[line width=2pt,draw=orange]
			table[x=x,y=y-test]
			{FI-CF100-sto-TT.dat};
			\addplot[line width=2pt,draw=cyan]
			table[x=x,y=y-test]
			{FI-CF100-sto-TF.dat};
			\addplot[line width=1pt,draw=black]
			table[x=x,y=y-test]
			{FI-CF100-sto-FF-min.dat};		
			\addplot[line width=1pt,draw=black,dotted]
			table[x=x,y=y-test]
			{FI-CF100-sto-FF-max.dat};	
			\addplot[line width=1pt,draw=orange]
			table[x=x,y=y-test]
			{FI-CF100-sto-TT-min.dat};
			\addplot[line width=1pt,draw=cyan]
			table[x=x,y=y-test]
			{FI-CF100-sto-TF-min.dat};
			\addplot[line width=1pt,draw=orange,dotted]
			table[x=x,y=y-test]
			{FI-CF100-sto-TT-max.dat};
			\addplot[line width=1pt,draw=cyan,dotted]
			table[x=x,y=y-test]
			{FI-CF100-sto-TF-max.dat};	
		\end{axis}
	\end{tikzpicture}
	\caption{Plots of SGD, PISGD, and PNSGD. The mean (thick solid), minimum (thin solid), and maximum (dotted) test set accuracy for R20C10 in $\FF_{15/20}$ (top), and R32C100 in $\FF_{17/24}$ (bottom) over 10 runs.} \label{T8}
\end{figure}

Let $\FF_{X/Y}\label{fi-ex}$ denote an $\FF$ with $\beta=2$, using X fractional bits and Y bits in total. Our use of QPyTorch followed closely the CIFAR10 Low Precision Training Example \citep{qdocu}. All weight and gradient rounding is done into $\FF_{X/Y}$, stochastic rounding is used throughout, no gradient accumulator is used, no gradient scaling is performed, and batch statistics are used to calculate the mean and variance for batch normalization. To determine the appropriate ratio of fractional bits, we were guided by the results of \citep{gupta2015}, and experimented with a majority of bits being fractional, given that in their experiments with $\FF_{X/16}$, the best accuracy occurred with X=14, with further improvement using $\FF_{16/20}$ \citep[Figures 1, 2, \& 3]{gupta2015}. The choice of the fixed-point environment $\FF_{X/Y}$ in each experiment was determined by finding the smallest Y which did not result in all algorithms collapsing to random guessing.

Let PNSGD$\label{SGD_V}$ denote PISGD with GN using $\alpha_k=0.05\hat{\eta}_k$. We plot the test set accuracy through time  for two experiments in Figure \ref{T8}: R20C10 in $\FF_{15/20}$ and R32C100 in $\FF_{17/24}$. In particular, the mean, minimum, and maximum accuracy over 10 runs are plotted. We observe that PNSGD is equal to or greater than SGD and PISGD in terms of the mean, minimum, and maximum accuracy. For the minimum accuracy, which we use as a measure of stability, PNSGD outperforms SGD and PISGD. It is observed that simple adaptive step sizes, without the need for any fine-tuning, can have a stabilizing effect on PISGD, making its training more robust to numerical error. 

Momentum and weight decay are typically used when training ResNet models \cite{he2016,idel18}. Compared to SGD, momentum requires storing another $d$-sized vector. Considering the GPU memory used by SGD to store model weights and gradients, by having to also store a momentum vector, this required GPU memory will increase by $50\%$. In settings with limited GPU memory, it may be more effective to allocate this memory to increasing the number of bits used in $\FF$, assuming that momentum and weight decay actually increase model accuracy when numerical error is present. Experiments were performed using momentum and weight decay with parameter values $R(0.9)$ and $R(1E-4)$ following \cite{he2016,idel18}. For R20C10, this resulted in all 10 runs collapsing to random guessing, with a final test set accuracy ranging within $[0.091,0.105]$. For R32C100, 3 runs collapsed to random guessing, with a final average test set accuracy of $0.426$, which is still significantly less than the final accuracy of $0.589$ using SGD (Figure \ref{T8}). This gives further evidence that perhaps ``simpler is better" when it comes to training with numerical error.

\subsection{Practical Usage of Assumption \ref{approx_assump}}
\label{testing_assum}

This section concludes by showing how inequality \eqref{errorineq} of Assumption \ref{approx_assump} can be verified empirically, since \eqref{errorupper} trivially holds in finite-precision environments. Instead of directly trying to choose $\FF$ which guarantees \eqref{errorineq}, a perhaps more practical approach is to view \eqref{errorineq} as a diagnostic tool when the algorithm is not converging as desired. Inequality \eqref{errorineq} could then be empirically tested to see if the numerical precision should be increased, instead of adjusting other parameters such as the step size or batch size.

In order to demonstrate this idea, the first 100 steps of the first run of the R20C10 experiments using PISGD were repeated using the same fixed-point environment $\FF_{15/20}$. In addition to the fixed-point model and its approximate stochastic gradient $\widehat{\nabla} F$, an FP32 model was used to (approximately) compute $\widetilde{\nabla} F$. After each training step, using the entire training set, the dot products
\begin{figure}[t]
	\centering
	\pgfplotsset{width=12cm,height=3.5cm,compat=1.3}
	\begin{tikzpicture}
		\pgfplotsset{scale only axis}
		\begin{axis}[
			axis y line*=left,axis x line*=bottom,xmin=0,xmax=101,ymin=-0.20,ymax=1.05,
			xtick={0,10,...,100},
			ytick={-0.20,0,...,1},
			xlabel=training steps,
			ylabel=$\hat{c}_1$,
			y label style={at={(axis description cs:-0.065,.5)},anchor=south},
			x label style={at={(axis description cs:.5,-0.3)},anchor=south}]
			\addplot[line width=1.5pt,draw=black]
			table[x=x,y=y]
			{FI-20-15.dat};\label{plot_20}		
			\addplot[line width=1.5pt,draw=orange]
			table[x=x,y=y]
			{FI-flat.dat};\label{plot_21}		
		\end{axis}
	\end{tikzpicture}
	\caption{For the first run of the R20C10 experiments, $\hat{c}^k_1$ was computed for the first 100 training steps, such that $\langle \EE[\widehat{\nabla} F(\bm{w}^k+\hat{\bm{u}}^k,\bm{\xi},\bm{b}^k)],\nabla f_{\alpha_k}(\bm{w}^k)\rangle=\hat{c}^k_1\|\nabla f_{\alpha_k}(\bm{w}^k)\|^2_2$, to empirically observe how well Assumption \ref{approx_assump} holds.} \label{F8}
\end{figure}

\begin{alignat}{6}
&\langle \EE[\widehat{\nabla} F(\bm{w}^k+\hat{\bm{u}}^k,\bm{\xi},\bm{b}^k)],\nabla f_{\alpha_k}(\bm{w}^k)\rangle\nonumber\\
=&\left\langle \frac{1}{N_T}\sum_{i=1}^{N_T}\widehat{\nabla} F(\bm{w}^k+\hat{\bm{u}}^k,\bm{\xi}^i,\bm{b}^k)],
\frac{1}{N_T}\sum_{i=1}^{N_T}\widetilde{\nabla} F(\bm{w}^k+\bm{u}^k,\bm{\xi}^i)\right\rangle\quad\text{and}\nonumber\\
& \|\nabla f_{\alpha_k}(\bm{w}^k)\|^2_2\nonumber\\
=&\left\langle \frac{1}{N_T}\sum_{i=1}^{N_T}\widetilde{\nabla} F(\bm{w}^k+\bm{u}^k,\bm{\xi}^i),
\frac{1}{N_T}\sum_{i=1}^{N_T}\widetilde{\nabla} F(\bm{w}^k+\bm{u}^k,\bm{\xi}^i)\right\rangle\nonumber 	
\end{alignat}
were computed, where $N_T=50,000$ is the size of the CIFAR-10 training dataset, from which the maximum value $\hat{c}^k_1$ was computed such that \eqref{errorineq} holds for $\bm{w}^k$. The values $\{\hat{c}^k_1\}\subset [-0.1383, 1.0382]$ are plotted in Figure \ref{F8}. Their mean value is $0.4522$, with 97 out of 100 being positive. If model training is performing poorly with a plot similar to Figure \ref{F8}, it is likely not the result of numerical error destroying the stochastic gradient signal. If instead for close to half of the steps $\hat{c}^k_1<0$, and their overall mean value is close to $0$, this would be strong evidence that numerical precision should be increased to improve the quality of the stochastic gradient. 

It is not always practical or even possible to do an exact expectation over the entire training set as described. In general, to verify the gradient accuracy for a chosen $\FF$, the same approach can be done using a large independent and identically distributed sample of data points $\{\bm{\xi}^i\}$. Readers are referred to \cite[Chapter 5]{lect_sp} for the theoretical analysis of estimating $\langle \EE[\widehat{\nabla} F(\bm{w}^k+\hat{\bm{u}}^k,\bm{\xi},\bm{b}^k)],\nabla f_{\alpha_k}(\bm{w}^k)\rangle$, $ \|\nabla f_{\alpha_k}(\bm{w}^k)\|^2_2$, and hence $\hat{c}^k_1$, using the sample average approximation approach.    

As a final comment, the fact that $\hat{c}^k_1<0$ for 3 iterations does not contradict nor disqualify the results proven in Section \ref{perturbedSGD} using Assumption \ref{approx_assump}, as these theorems are presenting sufficient conditions for convergence based on worst-case analysis for the class of functions satisfying their conditions. To the contrary, this experiment can be viewed as measuring how necessary Assumption \ref{approx_assump} is for the convergence of PISGD. Viewed as a sanity check, it is reassuring to observe that by choosing the smallest $Y$ in $\FF_{X/Y}$ which prevents the model from collapsing to random guessing, Assumption \ref{approx_assump} is satisfied for almost all iterations, and hence PISGD is clearly moving in a direction of descent on average over the iterations.   

\section{Conclusion}
\label{con}
This paper studied the convergence of Perturbed Iterate SGD with numerical error using adaptive step sizes. A new asymptotic convergence result to a Clarke stationary point is proven, as well as the non-asymptotic convergence to an approximate stationary point in expectation for stochastic Lipschitz continuous loss functions with error in computing stochastic gradients, as well as the SGD step itself. In terms of the permitted generality of adaptive step sizes, the importance of variance reduction through an increased mini-batch size was demonstrated, allowing non-asymptotic convergence guarantees for an unrestricted class of  adaptive step sizes, encompassing gradient normalization and clipping, over arbitrarily many iterations and to any desired accuracy. 

%

\bibliographystyle{tfs}
\bibliography{low_precision_oms}

\appendix
\section{Table of Notation}
\label{notation}
\setlength{\LTcapwidth}{\textwidth}
\begin{longtable}{llc}		
	\caption{Table of notation divided by section.}
	\label{t:9}	
	\endfirsthead
	\hline
	{\bf Symbol}&{\bf Description}&{\bf Page}\\		
	\hline				
	Section \ref{int}&&\\
	\hline	
	$f$&Loss function&\pageref{eq:0}\\		
	$F$&Stochastic loss function&\pageref{F}\\
	$\bm{\xi}$&Random vector argument of $F$&\pageref{xi}\\	
	$\bm{w}$&Decision variables of $f$&\pageref{w}\\		
	\hline		
	Section \ref{LipFunc}&&\\
	\hline				
	$L_0(\bm{\xi})$&Lipschitz constant of $F(\cdot,\bm{\xi})$ for almost all $\bm{\xi}$&\pageref{L0}\\	
	$Q$&$Q:=\EE[L_0(\bm{\xi})^2]$&\pageref{Q}\\
	$L_0$&$L_0:=\EE[L_0(\bm{\xi})]$&\pageref{EL0}\\		
	$B^p_{\epsilon}(\bm{w})$&$p$-norm $\epsilon$-closed ball centered at $\bm{w}$&\pageref{p-closed}\\	
	$B^p_{\epsilon}$&$B^p_{\epsilon}:=B^p_{\epsilon}(\bm{0})$&\pageref{p-closed-0}\\
	$\partial h$&Clarke subdifferential of a function $h$&\pageref{Clarke}\\
	$\partial^p_{\epsilon} h$&$p$-norm Clarke $\epsilon$-subdifferential of a function $h$&\pageref{p-n-e-Clarke}\\
	$\widetilde{\nabla} F$&Function equal to $\nabla F$ almost everywhere it exists&\pageref{approx_grad}\\		
	$\bm{u}$&Random vector uniformly distributed over $B^{\infty}_{\alpha}$&\pageref{uni-u}\\	
	$\alpha$&Radius of ball that $\bm{u}$ is sampled from&\pageref{alpha}\\	
	$f_{\alpha}$&$f_{\alpha}:=\EE[f(\cdot+\bm{u})]$&\pageref{f-alpha}\\	
	$L_1^{\alpha}$&Lipschitz constant of gradient of $f_{\alpha}$&\pageref{f-alpha-Lip}\\		
	\hline		
	Section \ref{finiteprec}&&\\
	\hline		
	$\FF$&A fixed-point arithmetic environment&\pageref{fi-e}\\
	$[n]_m$&$[n]_m:=[m,...,n]$&\pageref{int-seq}\\
	$[n]$&$[n]:=[n]_1$&\pageref{int-seq-1}\\
	$r$&Number of integer digits of a fixed-point number&\pageref{fi-r}\\
	$t$&Number of fractional digits of a fixed-point number&\pageref{fi-t}\\	
	$\beta$&Base of $\FF$&\pageref{base}\\				
	$d_i$&Value of the $i^{\text{th}}$ fractional digit of a fixed-point number&\pageref{fixed-frac}\\			
	$e_i$&Value of the $i^{\text{th}}$ integer digit of a fixed-point number&\pageref{fixed-int}\\		
	$\Lambda^-$&Smallest representable number in $\FF$&\pageref{f-s}\\
	$\lambda$&Smallest positive representable number in $\FF$&\pageref{f-sp}\\
	$\Lambda^+$&Largest representable number in $\FF$&\pageref{f-l}\\
	$\R_{\FF}$&$\R_{\FF}:=\{x\in\RR: \Lambda^-\leq x\leq \Lambda^+\}$&\pageref{range_FF}\\
	$\lfloor x\rfloor_{\FF}$&$\lfloor x\rfloor_{\FF}:=\max\{y\in \FF: y\leq x\}$&\pageref{x-floor}\\
	$\lceil x\rceil_{\FF}$&$\lceil x\rceil_{\FF}:=\min\{y\in \FF: y\geq x\}$&\pageref{x-ceil}\\ 
	$R$&Round to nearest or stochastic rounding&\pageref{x-round}\\
	\hline
	Section \ref{perturbedSGD}&&\\
	\hline	
	PISGD&Perturbed Iterate SGD&\pageref{eq:11}\\
	$\eta_k$&Step size of PISGD in the $k^{\text{th}}$ iteration&\pageref{ss}\\		
	$\hat{\eta}_k$& Deterministic component of $\eta_k$&\pageref{ss}\\
	$\psi_k$&Stochastic component of $\eta_k$&\pageref{ss}\\		
	$M$&Mini-batch size of PISGD&\pageref{MB}\\	
	$\widehat{\nabla} F$&An approximation of $\widetilde{\nabla} F$ due to numerical error&\pageref{asgrad}\\
	$\bm{b}$&Discrete random vector for computing $\widehat{\nabla} F$ using stochastic rounding&\pageref{b_sr}\\
	$\hat{\bm{u}}^k$& An approximation of $\bm{u}^k$ due to numerical error&\pageref{u-approx}\\
	$\hat{\bm{e}}^k$&Random vector modelling the error in computing a step of PISGD&\pageref{e-error}\\
	$\F_k$&$\F_k:=\sigma(\hat{\bm{u}}^j,\{\bm{\xi}^{j,i}\},\{\bm{b}^{j,i}\},\psi_j,\hat{\bm{e}}^j: j\in[k])$&\pageref{F-filt}\\ 
	$\G_k$&$\G_k:=\sigma(\hat{\bm{u}}^k,\{\bm{\xi}^{k,i}\},\{\bm{b}^{k,i}\},\psi_k)$&\pageref{G-sig}\\	
	$S^k$&$S^k:=\{\bm{w}^k,\hat{\eta}_k,\psi_k,M,\{\widehat{\nabla} F(\bm{w}^k+\hat{\bm{u}}^k,\bm{\xi}^{k,i},\bm{b}^{k,i})\}\}$&\pageref{Sk}\\		
	$\widehat{P}^k$&Distribution of $\hat{\bm{u}}$&\pageref{P-u}\\
	$V^k$&Support of random vector $\bm{b}^k$&\pageref{V-b}\\	
	$\widehat{\nabla} F^{k,i}(\bm{w}^k)$&$\widehat{\nabla} F^{k,i}(\bm{w}^k):=\widehat{\nabla} F(\bm{w}^k+\hat{\bm{u}}^k,\bm{\xi}^{k,i},\bm{b}^{k,i})$&\pageref{F-tilde-sim}\\
	$\widehat{\nabla} \overline{F}^k(\bm{w}^k)$&$\widehat{\nabla} \overline{F}^k(\bm{w}^k):=\frac{1}{M}\sum_{i=1}^M\widehat{\nabla} F^{k,i}(\bm{w}^k)$&\pageref{F-tilde-sim2}\\	
	$\Psi^L_k,\Psi^U_k$&$\PP(\Psi^L_k\leq\psi_k\leq\Psi^U_k|\F_{k-1})\stackrel{\text{a.s.}}{=}1$&\pageref{ss-bound}\\
	$\underline{\Psi}^U,\overline{\Psi}^U$&$\PP(\underline{\Psi}^U\leq\Psi^U_k\leq\overline{\Psi}^U)=1$&\pageref{ss-bound-bound}\\
	$\Delta_k$&$\Delta_k:=\Psi^U_k-\Psi^L_k$&\pageref{ss-Del}\\	
	$\Delta_{c_1}$&$0<\Delta_{c_1}\leq2c_1\Psi^U_k-\Delta_k$, by restricting $\Delta_k$ if necessary&\pageref{del_c1}\\		 
	\hline
	Section \ref{empirical}&&\\
	\hline
	R20C10&ResNet 20  trained on CIFAR-10&\pageref{r20}\\
	R32C100&ResNet 32 trained on CIFAR-100&\pageref{r32}\\	
	$\FF_{X/Y}$&$\FF$ with X fractional digits, Y digits in total, using stochastic rounding&\pageref{fi-ex}\\
	RGN&Restricted Gradient Normalization described in Algorithm \ref{alg1}&\pageref{alg1}\\
	GN&Gradient Normalization&\pageref{N_Grad}\\
	PNSGD&PISGD with GN&\pageref{SGD_V}\\		
	\hline	
\end{longtable}

\end{document}